\def \eps {\varepsilon}

\documentclass[11pt]{article}
\usepackage{amsfonts}
\usepackage{bbding}
\usepackage[dvips]{graphicx}
\usepackage{amsfonts,amssymb,amsmath,amsthm,amscd}
\usepackage{mathrsfs}
\usepackage{xcolor}
\usepackage{empheq}
\usepackage{adforn}
\usepackage{cancel}
\usepackage{xr}
\usepackage{mdframed}
\usepackage{mathdots}
\usepackage{enumerate}
\usepackage{lmodern}
\usepackage{mdframed}
\usepackage{stmaryrd}
\usepackage{blindtext}
\usepackage[all, knot]{xy}
\usepackage{leftidx}
\usepackage{accents}

\definecolor{slightblue}{rgb}{.8, .8, 1}
\definecolor{hair}{RGB}{100,225,190}
\definecolor{ruby}{RGB}{220,50,120}
\definecolor{grass}{RGB}{150,220,110}

\usepackage[colorlinks=true, pdfstartview=FitP, linkcolor=hair!80!black, citecolor=hair!80!black, urlcolor=hair!80!black]{hyperref}

\usepackage[T1]{fontenc}

\newtheorem{theorem}{Theorem}[section] \newtheorem{proposition}[theorem]{Proposition}
\newtheorem{lemma}[theorem]{Lemma}

\theoremstyle{definition} 
\newtheorem{definition}[theorem]{Definition}

\theoremstyle{remark} \newtheorem{remark}[theorem]{Remark} \numberwithin{equation}{section}
\numberwithin{figure}{section}

\newcommand{\R}{\mathbb{R}}

\renewcommand{\H}{\mathbb{H}}

\renewcommand{\P}{\mathbf{P}}
\newcommand{\E}{\mathbf{E}}

\newcommand{\U}{\mathbb{U}}

\usepackage{sidecap,subcaption}
\usepackage[section]{placeins}
\usepackage{wrapfig}
\usepackage{fullpage}

\usepackage[french,english]{babel}

\begin{document}

\title {Conditioning a Brownian loop-soup cluster\\ on a portion of its boundary}
\author{Wei Qian}
\date{}

\maketitle
\begin{abstract}
 We show that if one conditions a cluster in a Brownian loop-soup $L$ (of any intensity) in a two-dimensional domain by a portion $\partial$ of its outer boundary, then in the remaining domain, the union of all the loops of $L$ that touch $\partial$ satisfies the conformal restriction property while the other loops in $L$ form an independent loop-soup. This result holds when one discovers $\partial$ in a natural Markovian way, such as in the exploration procedures that have been defined in order to actually construct the Conformal Loop Ensembles as outer boundaries of loop-soup clusters. 
This result implies among other things that a phase transition occurs at $c=14/15$ for the connectedness of the loops that touch $\partial$. 
 
Our results can be viewed as an extension of some of the results in our paper \cite{Qian-Werner} in the following two directions: 
There, a loop-soup cluster was conditioned on its entire outer boundary while we discover here only part of this boundary.
And, while it was explained in \cite {Qian-Werner} that the strong decomposition using a Poisson point process of excursions that we derived there 
should be specific to the case of the critical loop-soup, we show here that in the subcritical cases, a weaker property involving the conformal restriction property nevertheless holds. 
\end{abstract}

\selectlanguage{french} 
\begin{abstract}
Dans le présent article, nous étudions certaines propriétés des amas de lacets browniens, dans une soupe de lacets browniens d'intensité $c$ dans un domaine du plan, pour toute intensité $c \le 1$. 
 
Notre principal résultat dit que si l'on découvre de manière Markovienne une portion $\partial$ du bord extérieur d'un tel amas, alors dans le domaine restant, la loi conditionnelle de l'union de tous les lacets dans $L$ qui touchent $\partial$ satisfait la propriété de restriction conforme tandis que les autres lacets dans $L$ forment une soupe de lacets indépendante. Ceci implique en particulier l'existence d'une transition de phase à $c=14/15$ pour la connectivité de l'ensemble des lacets qui touchent $\partial$.

Nos résultats constituent une extension de certaines résultats de notre papier \cite{Qian-Werner} dans les deux directions suivantes: Dans \cite{Qian-Werner}, un cluster de lacets est conditionné par son bord extérieur entier tandis que nous découvrons ici seulement une partie de ce bord. En outre, dans \cite{Qian-Werner}, nous expliquons que la description que nous donnons de l'ensemble des lacets qui touchent ce bord via un processus ponctuel de Poisson d'excursions est spécifique au cas de la soupe de lacets critique ($c=1$), nous montrons ici que dans les cas sous-critiques $c <1$, une propriété plus faible de restriction conforme reste néanmoins vraie.
\end{abstract}

\selectlanguage{english} 



\section{Introduction}

\subsection {General goal}

The simple conformal loop ensembles (CLE) form a natural conformally invariant class of random collections of non-intersecting random loops in a simply connected domain, and they are conjectured to arise as scaling limits of a number of discrete models (for instance, each CLE is the conjectural scaling limit of a discrete dilute $O(N)$ model -- in the special case of the critical Ising interfaces i.e. the $O(1)$ model and CLE$_3$, it has now been shown to hold, see \cite{BH} and the references therein). All these simple CLEs have been constructed by Sheffield and Werner as outer boundaries of Brownian loop-soup clusters in \cite {MR2979861}, which proved in particular that the SLE branching tree construction from \cite {MR2494457} did indeed construct a conformally invariant collection of disjoint loops. 
In the special critical case, this CLE (which is CLE$_4$) can also be constructed as level lines of the two-dimensional Gaussian free field (GFF), see Miller-Sheffield \cite {MS}, or \cite {ASW}. 
There exists also a coupling of the GFF with the other CLEs, but it is much more involved and somewhat less natural (it for instance relies on various choices, such as the choice of a root on the boundary of the domain), see for instance the comments in \cite {M-S-W, M-S-W2}. This suggests that understanding better the coupling of CLEs with subcritical loop-soups is actually of interest 
in order to understand better CLEs themselves.

In \cite {Qian-Werner}, we have described features of the conditional distribution of the loop-soup given the CLE that it defines  (i.e. given the outermost boundaries of clusters in the loop-soup). In particular, we have shown that conditionally on the CLE, the loop-soup splits into two different and conditionally independent parts: The set of loops that touch the CLE loops, and the set of loops that do not touch the CLE loops. In all CLEs, the latter collection turns out to be distributed like a loop-soup within the complement of the CLE loops. In the special critical case of CLE$_4$, related to the GFF, we have furthermore shown that the trace of the former collection of loops (that do touch the CLE$_4$ loops) is distributed exactly like a Poisson point process of Brownian excursions away from the CLE$_4$ loops. 
As explained at the end of \cite {Qian-Werner}, we do not expect this second result to be valid for subcritical loop-soups.

In the present paper, we present further results on the distribution of the loop-soup conditionally on part of the CLE loops. For instance, one considers explorations of the CLE from the boundary, where one traces along a CLE loop as soon as one hits it, and one stops this exploration while tracing a loop, and we will describe features of the conditional distribution of the loop-soup that are valid throughout the subcritical regime: 
\begin {itemize}
\item Again, the set of loops that intersect the discovered CLE pieces are independent from the set of loops that do not intersect the discovered CLE portions. The latter set is just distributed like a loop-soup 
in the remaining domain. 
\item The trace of the set of loops that do intersect the CLE portions are shown to satisfy a conformal restriction property. 
\end {itemize}

These results allow to shed some further light on the construction of CLEs via Brownian loop-soups.  
One interesting  outcome is the following phase transition that occurs at $c=14/15$ for the structure of the clusters. If we look at one outermost cluster in the loop-soup and denote by $\Gamma^b$ the set of loops in that cluster that touch its outer-boundary, then:
\begin{itemize}
\item
The loops in $\Gamma^b$ alone almost surely do not hook up into one single cluster when $c\in(14/15,1]$. However, they are almost surely hooked up into one single cluster by the loops  that are of positive distance from the outer-boundary of that outermost cluster.
\item The loops in $\Gamma^b$ alone do almost surely hook up into one single cluster when $c\in(0,14/15]$.
\end{itemize}
Note that the value $c=14/15$ is equal to $1-6/(m(m+1))$ for $m=9$, which is one of the special values of $c$ corresponding to the unitary irreducible highest weight representation in the Virasoro algebra (see \cite{PhysRevLett.52.1575}).  We do not know, however, whether it is a coincidence.

Another by-product of the results is a simple a posteriori explanation of the relation between conformal restriction measures, loop-soups and SLE curves that was derived by Werner and Wu in \cite {MR3035764} (see also \cite {MR2023758}). 

Our proofs will combine ideas from \cite {Qian-Werner} with tools developed in the derivation of the Markovian characterization of CLEs in \cite {MR2979861}.

\subsection{Background}

Before stating more precisely the main results of the present paper, we briefly survey some facts and earlier results about Brownian loop-soup clusters in two-dimensions and Conformal Loop Ensembles:

\medbreak
\noindent
{\bf The simple Conformal Loop Ensembles.}
In \cite{MR2979861}, Sheffield and Werner have defined the law of a simple CLE in $\U$ to be a collection  $\Lambda$ of countably many disjoint random simple loops that are all contained in $\U$ which satisfies the following two axioms:
\begin{itemize}
\item The law of $\Lambda$ is invariant under any conformal automorphism from $\U$ onto itself. Hence we can define a CLE in any simply connected domain $D$ by letting it be the image of $\Lambda$ under some given conformal map from $\U$ onto $D$.

\item 
For all  $A\subset\overline\U$ such that $\U\setminus A$ is simply connected, let $\tilde A$ be the union of $A$ with all the loops in $\Lambda$ that it intersects.
Conditionally on $\tilde A$, in each of the connected component $O$ of $\U\setminus \tilde A$, the loops  in $\Lambda$ that stay in $O$ are distributed like a CLE in $O$, independently of $\tilde A$ and of the loops outside of $O$.
\end{itemize}
They have also proved in \cite{MR2979861} that the simple CLEs form exactly a one-parameter family of measures indexed by $\kappa\in(8/3,4]$. 
We will come back to some of the ideas developed in this paper \cite{MR2979861} in the next paragraphs. 

\medbreak
\noindent
{\bf The Brownian loop-soups.}
The Brownian loop-soups have been introduced by Lawler and Werner in \cite{MR2045953}.
A Brownian loop-soup $\Gamma_D$  in a domain $D$ is a  Poisson point process  of unrooted Brownian loops  of intensity $c$, restricted to the loops that are contained in $D$. 
We choose the same renormalization of the measure $\mu^{\text{loop}}$ on unrooted Brownian loops  as in  \cite{MR2045953} (see for example (\ref{path-decomposition})).
The Brownian loop-soup is conformally invariant.
For any two domains $D_1, D_2$ such that there is a conformal map $\varphi$ from $D_1$ onto $D_2$, the law of $\varphi(\Gamma_{D_1})$ is equal to the law of $\Gamma_{D_2}$.

 When one studies loop-soups in a simply connected domain, by conformal invariance, one can choose to work the unit disk $\U$ or the upper half-plane $\H$.
A loop-soup $\Gamma$ in $\U$ contains almost surely infinitely many small loops (of diameter smaller than $\delta$) and only finitely many big loops.  A cluster of loops is defined to be an equivalent class  such that two loops $\gamma_1, \gamma_2$ in $\Gamma$ are in the same cluster if and only if there  is a finite sequence of loops $\gamma_3,\cdots,\gamma_n$ in $\Gamma$ such that $\gamma_i$ intersects $\gamma_{i+1}$ for $i=2,\cdots,n$, where $\gamma_{n+1}=\gamma_1$.
It is shown in \cite{MR2979861} that when $c\le 1$, the loop-soup $\Gamma$ of intensity $c$ contains a.s. infinitely many clusters and that for $c>1$, it contains a.s. one single cluster. 
In the present paper, we only study the case $c\le 1$. All our results are trivially true for the supercritical regime, in which no loop  touches $\partial \U$  (the outer boundary of the unique cluster).

It is furthermore proved in \cite {MR2979861} that the Brownian loop-soups of intensity $c\in(0,1]$ enable to construct simple CLEs. Almost surely, the closures of different clusters in a loop-soup do not intersect the boundary of the domain or intersect each other, and the collection of the outer boundaries of all the outer-most clusters in $\Gamma$ have the law of a CLE$_\kappa$ where $c(\kappa)=(3\kappa-8)(6-\kappa)/(2\kappa), \kappa\in(8/3,4]$. 
In Figure \ref{fig:loop-soup}, we see the loop-soup $\Gamma$ and the red interfaces which represent the outer boundaries of all the outer-most clusters in $\Gamma$.

\medbreak 
\noindent 
{\bf Markovian explorations of CLEs.} 
One idea of the paper \cite{MR2979861} is that the restriction property of CLE allows one to explore the collection $\Lambda$ from the boundary in a Markovian way that leads to the fact 
that the loops of an CLE are SLE loops. 
One example of such a Markovian exploration (is described in detail and used in \cite{MR2979861}) goes as follows:  One can discover the loops of a CLE $\Lambda$ in $\U$ in their order of appearance going from the right to the left along the segment $[-1,1]$. 

In fact, when one encounters a loop, instead of discovering the entire loop at once, one can choose to trace it in the counterclockwise direction until we close the loop.  
Then, one continues moving left along $[-1,1]$ and trace the (infinitely many) loops that still remain to be discovered. 
We denote by $\gamma$ the piecewise continuous curve (we choose for instance the right-continuous version) which is the concatenation of all the loops that we have discovered in this procedure.

One can for instance choose (but this is not important for our purposes because the actual time-parametrization will be irrelevant in our statements) to parametrize $\gamma$ up to the disconnecting time $S$ of the origin (that corresponds to the time at which one discovers the CLE loop surrounding the origin) using the radial normalization: For each $t>0$, let $K_t$ be the hull of $\gamma([0,t])$. Let $\varphi_t$
be the conformal map from $\U\setminus K_t$ onto $\U$ such that $\varphi_t(0)=0, \varphi_t'(0)>0$.
The radial parametrization of  $\gamma$ is such that $\varphi_t'(0)=\exp(t)$. See Figure \ref{fig:markov-exp}.

 One can also (and this would allow to work beyond $S$) parametrize using the chordal capacity of $K_t$ in $\U$ as seen from $-1$. 

\begin{figure}[h!]
\centering
\includegraphics[width=.95\textwidth]{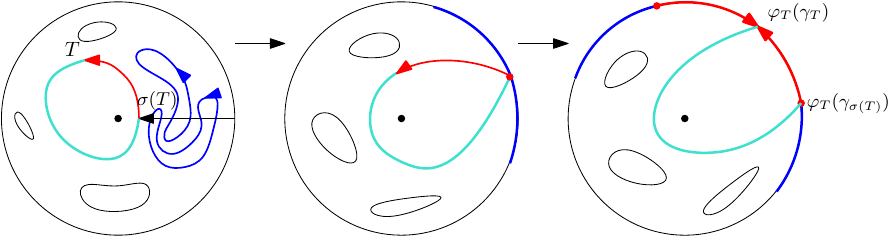}
\caption{The map $\varphi_T$ maps  $\U\setminus K_{T}$ onto $\U$.}
\label{fig:markov-exp}
\end{figure}

We call $T$ a stopping time for this exploration if it is a stopping time with respect to the filtration generated by $\gamma$. One way to interpret some of the results of \cite{MR2979861} goes as follows: Conditionally on $\gamma([0,T])$,
\begin{itemize}
\item If at time $T$, $\gamma$ is not tracing a CLE loop, then $\Lambda$ restricted to $\U\setminus K_T$ is just an independent CLE in $\U \setminus K_T$.
\item If at time $T$, $\gamma$ has only traced a part $\partial$ of some loop $\gamma_0$,  then we denote by $\sigma(T)$ the time at which we start tracing the loop $\gamma_0$. Then the conditional law of the rest of that loop is that of an SLE$_\kappa$ from $\gamma_T$ to $\gamma_{\sigma(T)}$ in $\U \setminus K_T$. 
\end{itemize}

Other examples of explorations of CLEs for which this property holds are described in \cite{MR2979861} or in \cite{MR2979861, MR3057185} (where it is proved that Sheffield's branching tree construction/exploration from \cite {MR2494457} gives indeed rise to a conformally invariant collection of loops satisfying the CLE axioms).  
\medbreak 

One question that we are going to answer in the present paper goes as follows: {\em
Suppose that a CLE has been constructed via a Brownian loop-soup as described above, and one discovers this CLE using a Markovian exploration that stops at some stopping time $T$. What can be said about the 
conditional distribution of the loop-soup?} Before giving an answer, we need to recall two further items.

\medbreak
\noindent
{\bf The decomposition of Brownian loop-soup clusters.}
Here we recall some of the main results in \cite{Qian-Werner} on the decomposition of loop-soup clusters. 

Let $\Gamma$ be a loop-soup in $\U$ and let  $\{\gamma_i, i\in I\}$ be the collection of the  outer boundaries of all the outer-most clusters of $\Gamma$, which is exactly the CLE coupled with it (the index set $I$ is therefore countably infinite). 
Let $O_i$ be the open domain enclosed by $\gamma_i$. Let $\Gamma_i$ be the set of all the loops in $\Gamma$ that are contained in $\overline O_i$.
Here we introduce the new notion of \emph{complete cluster}, referring to the union of an outer-most cluster $\theta$ with all the loops that are inside the domain enclosed by the outer boundary of $\theta$.
 The sets $\Gamma_i$ are the complete clusters of the loop-soup $\Gamma$.

\begin{figure}[h!]
\centering
\includegraphics[width=0.67\textwidth]{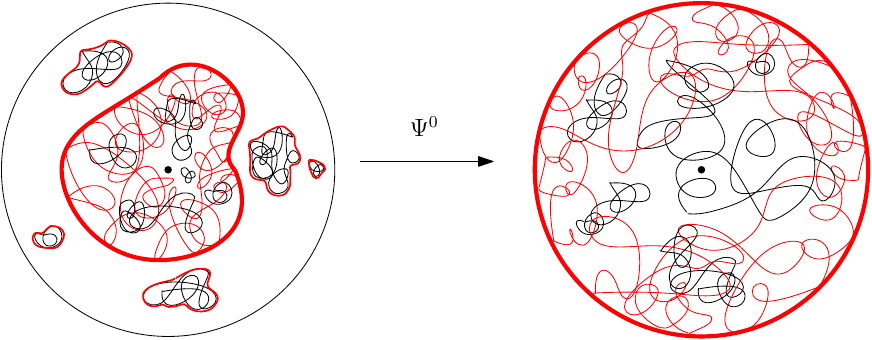}
\caption{The loops in $\Gamma^b$ are drawn in red and the loops in $\Gamma^i$ are drawn in black.}
\label{fig:intro2}
\end{figure}

\begin{figure}[ht!]
\centering
\includegraphics[trim = 20mm 60mm 20mm 60mm, clip, width=.65\textwidth]{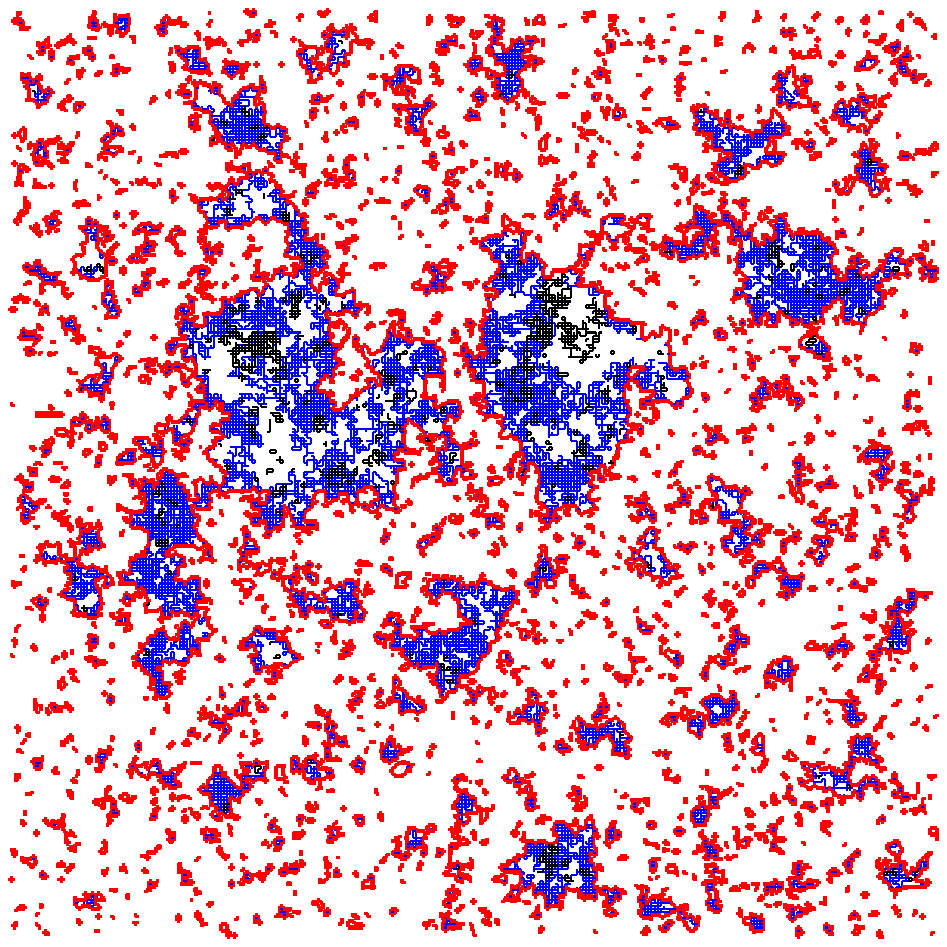}
\caption{A Brownian loop-soup, approximated by a random walk loop-soup on the square lattice with a cut-off of small loops  (this cut-off is necessary in the simulation because otherwise the loop-soup is dense in the domain). The approximation is justified by  \cite{MR2255196,MR3547746}. Also see  \cite{Lupu}.\\
The outer-boundaries of outer-most clusters are drawn in red. The boundary-touching loops are drawn in blue. The non-boundary-touching loops are drawn in black.}
\label{fig:loop-soup}
\end{figure}

It is shown in \cite{Qian-Werner} that  when one conditions $\Gamma$  on the collection $\{\gamma_i, i\in I\}$, the different complete clusters $\Gamma_i$ for $i\in I$ are independent of each other.
Let $\gamma_0\in\{\gamma_i, i\in I\}$ be the unique loop that surrounds the origin (the origin has no special role  due to conformal invariance). Let $O_0$ and $\Gamma_0$ be the domain and the complete cluster associated with $\gamma_0$.
Then we can define $\Psi_0$ to be the conformal map from $O_0$ onto $\U$ such that $\Psi_0(0)=0, \Psi_0'(0)>0$ (see Figure \ref{fig:intro2}).
The following lemma \cite[Lemma 2,3]{Qian-Werner} will be useful in this paper. 

\begin{lemma}[\cite{Qian-Werner}] \label{lem:Q-W}
The law of $\Psi_0(\Gamma_0)$ is independent of the collection $\{\gamma_i, i\in I\}$.
\end{lemma}
\begin{remark}\label{rem:p0}
We  denote the law of $\Psi_0(\Gamma_0)$ by $P_0$. The law $P_0$ is  invariant under all conformal automorphisms from $\U$ onto itself, due to the conformal invariance of the loop-soup.
\end{remark}

The set $\Psi_0(\Gamma_0)$ is the union of two  independent sets of loops:
the set $\Gamma^b$ of loops that touch $\partial\U$ and the set $\Gamma^i$ of loops that stay in the interior of $\U$ (also see Figure \ref{fig:loop-soup} for an illustration on a simulated Brownian loop-soup).
It is shown in \cite{Qian-Werner} that the two sets $\Gamma^b$ and $\Gamma^i$ are independent and the set $\Gamma^i$ is distributed like a loop-soup in $\U$. 
Moreover, for the critical intensity  $c=1$, the trace of the set $\Gamma^b$ is distributed like a Poisson point process of Brownian excursions in $\U$ having both end-points on $\partial\U$.

For small intensities, however, considerations on cut-points \cite{Qian-Werner} suggest that  $\Gamma^b$ should no longer be a Poisson point process of Brownian excursions.
It is our goal here to obtain more information on $\Gamma^b$ for $c\in(0,1)$. One result of the present paper (Proposition \ref {coro1}) is that, for all $c\in(0,1]$, the trace  of all the loops in $\Gamma^b$ satisfies a version of the conformal restriction property that we now describe.

\medbreak
\noindent
{\bf Conformal restriction.}
The conformal restriction property was first introduced and studied by Lawler, Schramm and Werner in \cite{MR1992830}. 
Note that this conformal restriction property is different from the restriction property of the CLE that we discussed earlier.

Let us first briefly recall the one-sided version of the chordal restriction property, following \cite {MR1992830}. 
Let $a,b\in\partial\U$ be two fixed boundary points.
We consider a class of random simply connected and relatively closed sets $K\subset \overline{\U}$ such that $K\cap \partial\U$ is equal to the arc from $a$ to $b$ in the clockwise direction. See Figure \ref{fig:two-point-restriction}.
Such a set (or rather, its distribution) is said to satisfy  \emph{chordal conformal restriction property} if the following two conditions hold:
\begin{itemize}
\item[(i)](conformal invariance) The law of $K$ is invariant under any conformal map from $\U$ onto itself that leave the boundary points $a$ and $b$ invariant.
\item[(ii)](restriction) For all closed sets $A\subset \overline\U$ such that $\U\setminus A$ is simply connected and that the distance from $A$ to the clockwise arc from $a$ to $b$ is strictly positive,
the conditional distribution of $\varphi_A(K)$ given $\{K \cap A=\emptyset\}$ is equal to the (unconditional) law of $K$, where $\varphi_A$ is any conformal map from $\U\setminus A$ onto $\U$ that 
leaves the points $a,b$  invariant (property (i) actually ensures that if this holds for one such map $\varphi_A$, then it holds also for any other such map). 
\end{itemize}
\begin{figure}[h]
\centering
\includegraphics[width=0.7\textwidth]{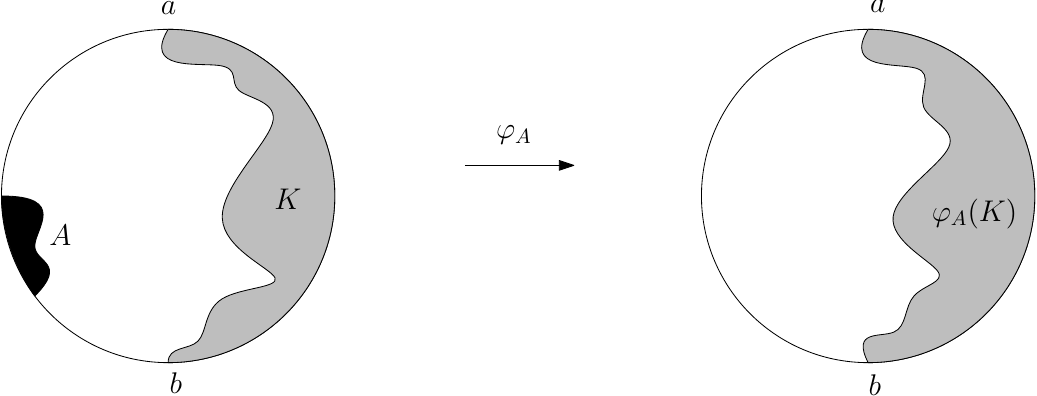}
\caption{
{Chordal restriction:  $\varphi_A$ is a conformal map from $\U\setminus A$ onto $\U$ that leaves $a,b$ invariant.}
The conditional law of $\varphi_A(K)$ 
given $\{K\cap A=\emptyset\}$ is equal to the (unconditional) law of $K$. }
\label{fig:two-point-restriction}
\end{figure}
It is proved in \cite{MR1992830} that measures that satisfy such one-sided chordal conformal restriction property can be characterized by a parameter $\alpha$ such that for all closed sets $A\subset \overline\U$ such that $\U\setminus A$ is simply connected and that the distance from $A$ to the clockwise arc from $a$ to $b$ is strictly positive,
\begin{align}\label{restriction}
\P(K\cap A=\emptyset)=\varphi_A'(a)^\alpha \varphi_A'(b)^\alpha.
\end{align}
Moreover, such measures exist if and only if $\alpha\ge 0$.
One-sided chordal restriction measures can be constructed in various ways, for example, as boundaries of Poisson point process of Brownian excursions, or using variants of the SLE$_{8/3}$ process.

The conformal restriction property can be naturally extended from the chordal case to other cases, such as the radial setting \cite{MR3293294} and the trichordal setting \cite{Qian-Trichordal}, for which the corresponding measures can also be characterized by a finite set of real parameters. 
However in other settings, the family of measures that satisfy the conformal restriction property can turn out to be much larger.

In the present paper, we will define and use another variant, the \emph{interior conformal restriction}: 
\begin{definition}\label{def:interior-restriction}
We consider a class of relatively closed random sets $K\subset\overline\U$. See Figure \ref{fig:interior-restriction}. Such a set (or rather, its distribution) is said to satisfy  \emph{interior restriction} if the following two conditions hold:
\begin{itemize}
\item[(i)](conformal invariance) The law of $K$ is invariant under any conformal map from $\U$ onto itself.
\item[(ii)](restriction) Let $U_1, U_2\subset\U$ be two connected domains such that there is a conformal map $\varphi$ from $U_1$ onto $U_2$ and that the probability of $\{K \subset U_1\}$ is non zero.  
Let $P_1$ be the conditional distribution of $K$ given $\{K \subset U_1\}$ and let $P_2$ be the conditional distribution of $K$ given $\{K \subset U_2\}$. Then the image of $P_1$ under $\varphi$ is equal to $P_2$ (property (i) actually ensures that if this holds for one such map $\varphi$, then it holds also for any other such map). 
\end{itemize}
\end{definition}

\begin{figure}[h]
\centering
\includegraphics[width=0.7\textwidth]{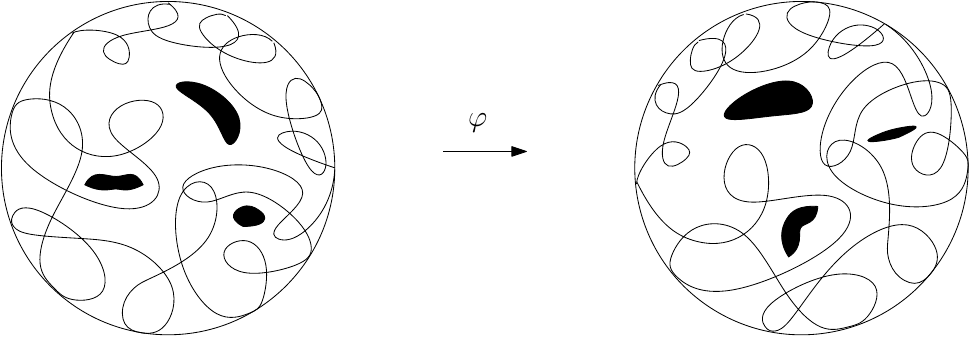}
\caption{Interior restriction. Here $U_1$ is $\U$ minus the black regions  in the left picture, and $U_2$ is the one in the right picture. The conformal map $\varphi$ sends $U_1$ onto $U_2$. The set $K$ is represented by the random curves.}
\label{fig:interior-restriction}
\end{figure}

Examples of measures that satisfy this interior restriction are given by Poisson point processes of Brownian excursions away from the boundary.

One result that we will prove in the present paper is the following: 

\begin{proposition}\label{coro1}
The set $\Gamma^b$ (defined after Lemma \ref {lem:Q-W})  satisfies interior restriction.
\end{proposition}

We will not address in the present paper the issue of how large the class of interior restriction measures is.

Let us now define the notion of \emph{chordal interior restriction}, which is a variant of the interior restriction property with two additional marked points. 
\begin{definition}\label{def:chordal-interior-restriction}
For $a,b\in\partial\U$, we consider a class of relatively closed random sets $K_{a,b}\subset\overline\U$ such that $K$ do not touch the clockwise arc from $b$ to $a$. The family of sets $K_{a,b}$ (or rather, their distributions) is said to satisfy  \emph{chordal interior restriction} if the following two conditions hold:
\begin{itemize}
\item[(i)](conformal invariance) For any conformal map $\varphi$ from $\U$ onto itself, the law of $\varphi(K_{a,b})$ is equal to the law of $K_{\varphi(a),\varphi(b)}$.
\item[(ii)](restriction) Let $U_1, U_2\subset\U$ be two connected domains such that there is a conformal map $\varphi$ from $U_1$ onto $U_2$
 and that the probabilities of $\{K_{a,b} \subset U_1\}$ is non zero.  
Let $P_1$ be the conditional distribution of $K_{a,b}$ given $\{K_{a,b} \subset U_1\}$ and let $P_2$ be the conditional distribution of $K_{\varphi(a),\varphi(b)}$ given $\{K_{\varphi(a),\varphi(b)} \subset U_2\}$. Then the image of $P_1$ under $\varphi$ is equal to $P_2$.
\end{itemize}
\end{definition}
Note that chordal interior restriction implies chordal restriction. 

\medbreak
\noindent
{\bf The relation between restriction measures, CLE and SLE.}
Finally, we state a result by Werner and Wu \cite {MR3035764} which will turn out to be closely related to our result as it will be explained in the next section.
\begin{figure}[h]
\centering
\includegraphics[width=0.28\textwidth]{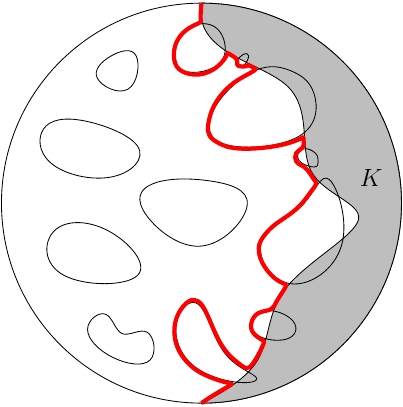}
\caption{The sample $K$ of a restriction measure, the CLE$_\kappa$ and the SLE$_\kappa (\rho)$.}
\label{result-hao}
\end{figure}

The statement goes as the following (see Figure \ref{result-hao}): Let $K$ be the sample of a one-sided restriction measure in $\U$ with exponent $\alpha$ along the clockwise arc from $i$ to $-i$. Let $\Gamma$ be an independent CLE$_\kappa$ in $\U$. Then the union of $K$ with all the loops in $\Gamma$ that it intersects has an outer boundary which is distributed like a SLE$_\kappa(\rho)$ curve. The relation between $\alpha$ and $\rho$ is given by $\alpha=(\rho+2)(\rho+6-\kappa)/(4\kappa)$.


\subsection{Main result}
We can now finally state our main result and make some comments.
Let $\Gamma$ be a loop-soup in $\U$ with intensity $c \le 1$ in the unit disk. This defines a CLE that we explore in a Markovian way as described above (for instance exploring the CLE loops that intersect $[-1,1]$). 
Suppose that $T$ is some finite stopping time for the Markovian exploration and that at this stopping time $T$, 
one is almost surely currently in the middle of tracing a CLE loop. One can then define as before the set $K_T$ and the conformal map $f_T$
 from $\U \setminus K_T$ onto $\U$ that maps 
$\gamma_T$, $-1$ and $\gamma_{\sigma (T)}$ respectively to $i$, $-1$ and $-i$ (see Figure \ref{fig:intro}).

As we have already explained, the conditional law of $\gamma$ from time $T$ up to the time at which it completes the loop that it is tracing at time $T$ is a SLE$_\kappa$ in $\U \setminus K_T$. We denote by $\eta$ the image under $f_T$ of this part of $\gamma$, which is a SLE$_\kappa$ in $\U$ from $-i$ to $i$.
We define $L$ to be the collection of loops in the loop-soup that do touch $\gamma [\sigma(T), T]$.  We denote by $\tilde \Gamma$ the collection of loops in $\U\setminus K_T$ that are not in $L$. We denote by $f_T(L)$ and $f_T(\tilde\Gamma)$ their corresponding images by $f_T$.

\begin{figure}[h!]
\centering
\includegraphics[width=0.72\textwidth]{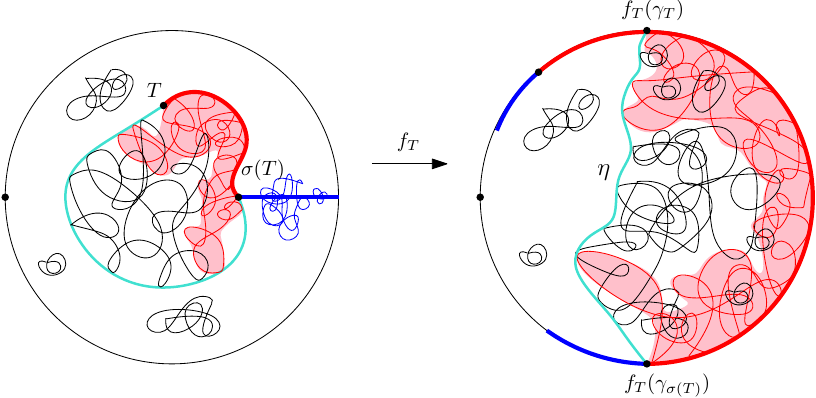}
\caption{Sketch of the loops in $L$  (red), the loops in $\tilde\Gamma$ (black), the loops in $K_T$ (blue) and their images under $f_T$. }
\label{fig:intro}
\end{figure}

\begin{theorem}\label{main theorem}
The sets $f_T(L)$ and $f_T(\tilde\Gamma)$ are independent and they are furthermore independent from $\gamma[0,T]$. Moreover, they satisfy the following property:  
\begin {itemize}
\item $f_T(\tilde\Gamma)$ is a loop-soup of intensity $c$ in $\U$.
\item The trace of the union of loops in $f_T(L)$ satisfies the one-sided chordal restriction (and a stronger restriction property that we call chordal interior restriction) with exponent $\alpha=(6-\kappa)/(2\kappa)$. 
\end {itemize}
\end{theorem}

The main part of this theorem is the second part on the conformal restriction property of the set $f_T(L)$. It is not surprising that one can deduce, using similar arguments as in \cite{Qian-Werner} and taking some care of the exploration process, that  $f_T(L)$ and $f_T(\tilde\Gamma)$ are independent from each other and from $\gamma[0,T]$, and that $f_T(\tilde \Gamma)$ is distributed as a loop-soup.
We also know by \cite{MR2979861} that $\eta$ has the law of a SLE$_\kappa$. Relating this to Werner and Wu's result \cite{MR3035764}, one would naturally guess that (the filling of) $f_T(L)$ should satisfy the one-sided conformal restriction property. 
We prove that it is indeed the case and also show a stronger version of the conformal restriction property for $f_T(L)$.  
Note that the relation between $\kappa$, $\alpha = ( 6 - \kappa)/ ( 2 \kappa)$ and $c(\kappa) = (3 \kappa -8) \alpha$ is the ``usual one'' that appears in the formulation of the SLE conformal  restriction property defect in \cite {MR1992830} and in the 
Conformal Field Theory framework ($\alpha$ is there the highest-weight of the degenerate representation of the Virasoro algebra with central charge $c$). 

Note that in the special case $\kappa=4$, it is actually possible to adapt fairly directly the arguments of  \cite{Qian-Werner} to see that the trace of $f_T(L)$ has the same law as the trace of a Poisson point process of Brownian excursions in $\U$ away from the right-half circle; this indeed then implies the conformal restriction property described in our theorem.

When $\kappa \not=4$, our result does not describe fully the conditional law of $K$, but it does nevertheless imply a number of its features. 
One interesting fact that is worth stressing is the following phase transition that occurs at $c= 14/15$ for the structure of loop-soup clusters. Consider the set $\Gamma^b$ as defined after Lemma \ref {lem:Q-W}. 
Then, 
\begin{itemize}
\item
When $c\in (14/15,1]$, which corresponds to $\kappa\in(18/5,4]$ and $\alpha\in[1/4,1/3)$, the union of loops in $f_T(L)$ almost surely has cut-points seen from $-1$, due to the fact that a one-sided chordal restriction measure of exponent $\alpha$ almost surely has cut-points when $\alpha<1/3$ (see \cite{MR1992830, MR2178043}).
Since there is no cut-point of a Brownian motion that is also a double point (\cite{MR1062056}), the cut-points of the set $f_T(L)$ cannot belong to any loop. 
Therefore, the loops in $f_T(L)$ alone almost surely do not hook up into a single cluster. This implies that the loops in $\Gamma^b$  do almost surely not hook up into  a single cluster by themselves either. It is only thanks to the loops that are at positive distance from the boundary of the cluster that the loops in $\Gamma^b$ are connected into the same cluster. As this occurs with probability one, this shows 
that for any $\eps$, the inside loops of size smaller than $\epsilon$ are alone almost surely sufficient to connect all the loops of $f_T (L)$ into a single cluster. 

\item
When $c\in(0,14/15]$, which corresponds to $\kappa\in(8/3,18/5]$ and $\alpha\in[1/3, 5/8)$,   then the set $f_T(L)$ has no cut-point seen from $-1$. This implies 
that the loops in $\Gamma^b$ alone do almost surely hook up into a single cluster. Indeed, if this were not the  case, then the set $f_T(L)$ would have cut-points with positive probability.

\end{itemize}


Let us also mention that the measure $\mu^{\text{loop}}$ on a single Brownian loop can in some sense be seen as the limit of a loop-soup with an intensity tending to zero. In  Section \ref{sec2}, we will prove a similar decomposition-restriction result for one Brownian loop. This will be our Proposition \ref{prop:loop} and it can be thought of as the limit as $c \to 0$ of our other results.

Finally, the main body of this paper consists of Sections \ref{sec:one-point} and \ref{sec:glued}. We will spend some time at the beginning of both sections to explain the CLE exploration procedures (developed in \cite{MR2979861}) and ensure that the statements can be carried through to the loop-soup setting. 
We then proceed to the decompositions of the loop-soups in the different stages of the exploration process. 
We deduce the restriction property of certain subsets of the loop-soup and ensure that this property can be carried through the successive stages.
While the exploration process concerns mainly the outer boundaries of the clusters, the decompositions thereafter depend mainly on the internal structure of the loop-soups.

\section{Notations}\label{sec:notations}
Here we give some basic definitions and  notations that we will use throughout the whole paper (nevertheless each section may also contain some locally defined notations).

\smallskip
\noindent
{\bf Important domains and sets.}
Let $\H$ denote the upper-half plane, $\U$ the unit disk and $\mathbb{C}$ the whole plane.
In the upper half-plane setting, for $a\in\R$, let $D(a,\eps)$ be the half-disk contained in $\H$, centered at $a$ with radius $\eps$. In the unit disk setting, for $a\in\partial\U$, by an abuse of notation, we continue to denote by $D(a,\eps)$ the open set in $\U$ such that its image by the map $\varphi:z\mapsto i(1+z)(1-z)$ is equal to the half-disk in $\H$ centered at $\varphi(a)$ of radius $\eps$.
Let $\mathcal{A}$ be the set of all relatively closed sets $A\subset\H$ such that $d(0,\H\setminus A)>0$. As opposed to the presentation in the introduction that was mostly done in the unit disc, we will in fact mostly choose to work in the upper half-plane. 

\smallskip
\noindent
{\bf Filling of a set.} We denote by $F(K)$ the filling of a set $K\subset\mathbb{C}$, which is the complement of the unbounded connected component(s) of $\mathbb{C}\setminus K$.

\smallskip
\noindent
{\bf The event $z\in \theta$.}
When we say that a collection $\theta$ of loops {\bf encircles} or {\bf surrounds} a point $z$, we mean that $z$ is inside the filling of the union of the loops in that collection. We denote this event by $z\in\theta$.

\smallskip
\noindent
{\bf The collection of loops $\Gamma$ as a subset of the plane.}
Let $\Gamma, \tilde \Gamma$ be two collections of loops and let $A$ be some subset of the complex plane.
 We mean by $\Gamma\cap \tilde\Gamma$ the collection of loops which is the intersection of $\Gamma$ and $\tilde\Gamma$. However, sometimes we  abuse the notation by writing things such as $\{\Gamma\cap A=\emptyset\}$, where $\Gamma$ is identified to a subset of the complex plane which is equal to the union of all the loops in $\Gamma$. 
The meaning of $\Gamma$ will be clear from context.

\smallskip
\noindent
{\bf Complete cluster in $\Gamma$.}
Given a collection $\Gamma$ of loops in some simply connected domain $D$, a complete cluster in $\Gamma$ is defined to be a collection of loops that consists of the loops in an outer-most cluster $C$ of $\Gamma$ and of all the loops that are contained in the domain enclosed by the outer boundary of $C$.
When we say that $\theta$ is a complete cluster without giving the bigger collection $\Gamma$, we mean that $\theta$ is a complete cluster in $\theta$.

\smallskip
\noindent
{\bf Loop configurations in $\H$.}
Let $\theta$ be a complete cluster contained in $\H$.  Let $\Gamma$ be a collection of loops that are all contained in the domain $\H\setminus F(\theta)$. Then we say that $(\theta,\Gamma)$ is  a \emph{loop configuration in $\H$}.
We often make an abuse of notation by identifying the pair $(\theta,\Gamma)$ with the collection $\theta\cup\Gamma$. The meaning will be clear in the circumstances.

\smallskip
\noindent
{\bf One-point pinned configuration.} 
A loop configuration $(\theta,\Gamma)$ in $\H$ is said to be a \emph{one-point pinned configuration} if $\overline\theta\cap\R=\{0\}$.

\smallskip
\noindent
{\bf Two-point pinned configuration.}
A loop configuration $(\theta,\Gamma)$ in $\H$ is said to be a \emph{two-point pinned configuration} if $\overline\theta\cap\R=\{0,1\}$.

\smallskip
\noindent
{\bf Glued configuration.}
A loop configuration $(\theta,\Gamma)$ in $\H$ is said to be a \emph{glued configuration} if $\overline\theta\cap\R=[0,1]$.

\smallskip
\noindent
{\bf The collections of loops $C_D$ and $C^D$.}
For a collection $C$ of loops, $C_D$ with subscript $D$ refers to the collection of loops in $C$ that are in $D$ and $C^D$ with  superscript $D$ refer to the collection of loops in $C$ that intersect $D$.

\smallskip
\noindent
{\bf The collection of loops $f(\Gamma)$.}
Let $O_0\subset O_1$. Let $f$ be a conformal map from an open domain $O_0$ onto another open domain $O_2$. Let $\Gamma$ be a collection of loops in $\overline O_1$. We denote by $f(\Gamma)$ the image of the set $\Gamma_0$ under $f$, where $\Gamma_0$ is the collection of loops in $\Gamma$ that are contained in $\overline O_0$.

\smallskip
\noindent
{\bf Important measures.}
Let $P^0$ denote the measure on the collection of interior loops in a complete cluster, see Remark \ref{rem:p0}.
Let $\P_D$ denote the measure of a loop-soup in $D$. Let $\P^D$ denote the measure of a loop-soup in $\H$ restricted to those loops that touch $D$.
Other  measures such as $\nu(i), \nu, \bar\nu, \nu_A, \bar\nu_A, \rho_A, \bar\rho_A$ and so on will be defined later in the paper.
For all random collections of loops $\pi$, we denote by $P\pi$ the measure of $\pi$.

\smallskip
\noindent
{\bf The map $\Phi$.}
For a simple loop $\gamma$, a point $z$ enclosed by the loop, and a closed set $K\subset \overline\U$, let $\Phi$ be the map that maps $(\gamma, z, K)$ to the set $\phi(K)$ where $\phi$ is the conformal map from  $\overline\U$ onto the closure of the domain enclosed by $\gamma$ such that $\phi(0)=z, \phi'(0)>0$. 
In this paper, we will frequently apply the map $\Phi$ to a random triple $(\gamma, z, K)$  where the law of $K$ is invariant under all conformal automorphism from $\U$ onto itself. We will further choose $\gamma,z$ independently of $K$ so that the law of $\Phi(\gamma,z,K)$ does not depend on the choice of $z$. In this case, when we are talking about the laws, we do not need to give precision on the point $z$ and we can talk about the image under $\Phi$ of the law of $(\gamma, K)$.

\smallskip
\noindent
{\bf The map $\Phi^{\text{up}}$.} We define the map $\Phi^{\text{up}}$ which upgrades a complete cluster to a configuration in $\H$. Let $\theta$ be a complete cluster and let $\Gamma$ be some collection of loops in $\H$. Let $\Phi^{\text{up}}(\theta,\Gamma)$ be equal to the configuration $(\theta, \tilde\Gamma)$, where $\tilde \Gamma$ is the restriction of $\Gamma$ to the loops that are contained in the domain $\H\setminus F(\theta)$.

\section{The case of a single Brownian loop}\label{sec2}
In this section, as a warm-up, we discuss the case of a single Brownian loop, which can be seen as the limit of a loop-soup cluster when the intensity of the loop-soup tends to zero.  Our main result 
will be Proposition~\ref{prop:loop}, which can be interpreted as the $c\to 0^+$ limit of Proposition \ref{coro1}. The case of a single Brownian loop is much simpler than the loop-soup case because the cluster and the boundary-touching loops in the cluster are just the single Brownian loop itself.

We need to be careful since the Brownian loop measure $\mu^{\text{loop}}$ is an infinite measure. However it is not difficult to find  good analogies for the Brownian loop case with the  loop-soup case.
For example, the exploration process for the loop-soup and the definition of the one-point pinned complete cluster 
is analogous with the following path-decomposition of the Brownian loop measure \cite[Proposition 7]{MR2045953}
\begin{align}\label{path-decomposition}
\mu^{\text{loop}}=\frac{1}{\pi}\int_{-\infty}^\infty \int_{-\infty}^\infty \mu^{\text{bub}}_{\H+iy}(x+iy) dx dy,
\end{align}
where  $\mu^{\text{bub}}_{\H+iy}(x+iy)$ is the Brownian bubble measure rooted at $x+iy$, and contained in the upper half-plane above the horizontal line of height $y$.

In the following, we will first work with the probability measure $P^{\text{\text{bub}}}$ which is equal to the infinite measure $\mu^{\text{bub}}_\U(1)$ on the Brownian bubbles in $\U$ rooted at $1$ restricted to the loops that encircle the origin and then renormalized. Later, we will generalize the results to the $\mu^{\text{loop}}$ measure.
Let $\theta$ be a Brownian bubble with the law  $P^{\text{\text{bub}}}$. Let $\gamma$ be the outer boundary of $\theta$. Let $f_\gamma$ be the conformal map sending the domain enclosed by $\gamma$ onto $\U$ that leaves the points $0,1$ invariant.
We first prove the following lemma.

\begin{lemma}\label{lem0}
The sets $f_\gamma(\theta)$ and $\gamma$ are independent from each other.
\end{lemma}
\begin{proof}
This is a consequence of the (boundary) conformal restriction property of the measure $P^{\text{\text{bub}}}$.
For all closed sets $A\subset\U$ such that $\U\setminus A$ is simply connected, $d(0,A)>0, d(1,A)>0$, conditionally on the event $\{\theta\cap A=\emptyset\}$, the law of $\varphi_A(\theta)$ is the same as the (unconditional) law of $\theta$, where $\varphi_A$ is the conformal map from $\U\setminus A$ onto $\U$ that leaves the points $0,1$ invariant.
Note that on the event  $\{\theta\cap A=\emptyset\}$, the sets $f_\gamma(\theta)$ and $f_{\varphi_A(\gamma)}(\varphi_A(\theta))$ are the same.
Therefore  $f_\gamma(\theta)$ conditionally on $\{\theta\cap A=\emptyset\}$ has the same law as $f_{\varphi_A(\gamma)}(\varphi_A(\theta))$  conditionally on $\{\theta\cap A=\emptyset\}$ which is equal in law to the (unconditional)  $f_\gamma(\theta)$. Since this is true for all such $A$, the lemma is proved.
\end{proof}

We denote by $P^{\text{int}}$ the law of  $f_\gamma(\theta)$. The law $P^{\text{int}}$ depends a priori on the marked points $0$ and $1$. However, we will show that it is not the case and that $P^{\text{int}}$ is invariant under all conformal automorphisms from $\U$ onto itself.

\begin{figure}[h!]
\centering
\includegraphics[width=0.95\textwidth]{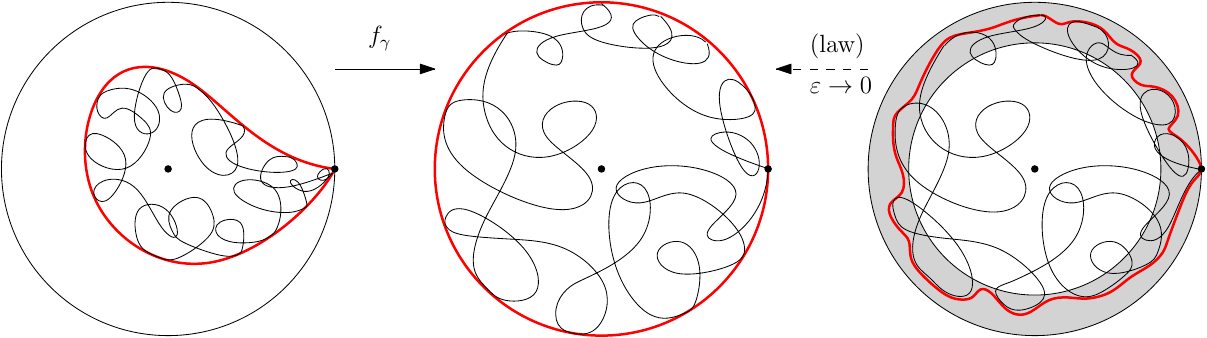}
\caption{The loop $\theta$ conditioned on the event $E_\eps $ converges in law to $f_\gamma(\theta)$.}
\label{fig:loop}
\end{figure}

\begin{lemma}\label{lem01}
The law $P^{\text{int}}$ is invariant under all conformal automorphisms from $\U$ onto itself that leave the point $1$ invariant.
\end{lemma}
\begin{proof}
Let $E_\eps$ be the event that $\gamma$ is contained in the ring $\U\setminus (1-\eps)\U$ and winds around the origin. The law of $\theta$ conditioned on the event $E_\eps$ converges to the law $P^{\text{int}}$ under the Hausdorff metric (see Figure \ref{fig:loop}), because of Lemma \ref{lem0} and the fact that the conformal map $f_\gamma^{-1}$ on the event $E_\eps$ is close to the identity map under the uniform distance. 

Let $z\in\U$ be some point different from the origin and let $f_z$ be the conformal map from $\U$ onto itself that sends the points $z,1$ to $0,1$. Choose some $\eps>0$ such that $|f_z(0)|, |z|<1-\eps$. Then the image under the map $f_z$ of  the law of $\theta$  conditionally on $E_\eps$ is equal to the law of $\theta$ conditionally on the event that $\gamma$ is contained in the domain $\U\setminus f_z((1-\eps)\U)$, due to conformal invariance of the measure $\mu^{\text{bub}}$ on Brownian bubbles. However, the law  of $\theta$ conditionally on the event that $\gamma$ is contained in the domain $\U\setminus f_z((1-\eps)\U)$ also converges to $P^{\text{int}}$ as $\eps\to 0$. 
Therefore, by taking the $\eps\to 0$ limit, we obtain that the law $P^{\text{int}}$ is invariant under $f_z$.
Since this is true for all $z$, we have proved the lemma.
\end{proof}

\begin{lemma}\label{lem02}
The law $P^{\text{int}}$ is invariant under all conformal automorphisms from $\U$ onto itself that exchange the points $e^{i\alpha}$ and $1$ for any $\alpha\in(0,2\pi)$.
\end{lemma}
\begin{proof}
We continue to use the notations of the previous lemma. Let $\theta$ be a Brownian bubble sampled with the law $P^{\text{\text{bub}}}$. We explore $\theta$ along the radius $[e^{i\alpha}, 0]$ from the boundary to the interior. Let $T=\sup\{t>0, t e^{i\alpha}\in\theta\}$ be the first time that we encounter the loop $\theta$. Then conditionally on $T$, the loop $\theta$ is distributed like the union of two independent Brownian excursions in the domain $\U\setminus [t e^{i\alpha}, e^{i\alpha}]$, one from $1$ to $te^{i\alpha}$ and the other from $te^{i\alpha}$ to $1$ (see \cite{MR2045953}).
The image $f_\gamma(\theta)$ then has two marked points $1$ and $f_\gamma(T e^{i\alpha})$ that play symmetric roles. The law of $f_\gamma(\theta)$ is invariant under all conformal automorphisms from $\U$ onto itself that leave the points $1, f_\gamma(T e^{i\alpha})$ invariant or exchange the two points.
We know that the law of $\theta$ conditioned on the event $E_\eps$ converges to the law $P^{\text{int}}$.
In the meantime, the point $f_\gamma(T e^{i\alpha})$ conditionally on $E_\eps$ converges a.s. to $e^{i\alpha}$ as $\eps\to 0$. Therefore we have proved the lemma.
 \end{proof}

The results above readily imply the following lemma. Let $\mu^{\text{bub,ext}}$ be the infinite measure on the outer boundary of the Brownian bubble in $\U$ rooted at $1$
(which is equal to the measure on SLE$_{8/3}$ rooted loops, up to normalizing constant, see \cite{MR2350053}). 
Recall the map $\Phi$ defined in \S\,\ref{sec:notations}.
\begin{lemma}
The measure $P^{\text{int}}$ is invariant under all conformal automorphisms from $\U$ onto itself. The measure $\mu^{\text{bub}}_\H(1)$ is equal to the image of $\mu^{\text{bub,ext}}\otimes P^{\text{int}}$ under the map $\Phi$.
\end{lemma}

Before going to the Brownian loop measure $\mu^{\text{loop}}$, we also prove the following proposition.
\begin{proposition}\label{prop2.5}
The measure $P^{\text{int}}$ satisfies interior restriction (as in Definition \ref{def:interior-restriction}).
\end{proposition}
\begin{proof}
It has been proved in \cite{MR2045953} that the measure $P^{\text{\text{bub}}}$ satisfies interior restriction with the additional marked points $0$ and $1$. 
Using the fact that $P^{\text{\text{bub}}}$ conditioned on the event $E_\eps$ converges to $P^{\text{int}}$, we can verify that $P^{\text{int}}$ satisfies the same interior restriction property with marked points $0,1$.  (We omit the details here. The reader can also see \cite[Lemma 7.10]{Qian-Trichordal} for a proof in the similar spirit.)
The dependence on the points $0$ and $1$ can be ruled out, since we have already proved that $P^{\text{int}}$ is invariant under all conformal automorphisms from $\U$ onto itself.  
\end{proof}

Now we can readily extend this result to the measure  $\mu^{\text{loop}}$, thanks to the path decomposition (\ref{path-decomposition}) of the Brownian loop measure.
Let $\mu^{\text{ext}}$ be the infinite measure on the outer boundary of the Brownian loop \cite{MR2350053} (which is equal to the measure of SLE$_{8/3}$ loops). 
Then we have the following proposition.

\begin{proposition} \label{prop:loop}
There exist a probability measure $P^{\text{int}}$ on closed sets $K\subset\U$ which satisfies interior restriction (hence is invariant  under all conformal automorphisms from $\U$ onto itself), such that  $\mu^{\text{loop}}$ is the image under $\Phi$ of $\mu^{\text{ext}}\otimes P^{\text{int}}$. 
\end{proposition}

Let us now briefly comment on the decomposition of the Brownian loop, given part of its boundary, as this will provide some motivation and insight into the more general case of loop-soup clusters (the following comment corresponds to the $c \to 0$ limit of loop-soup clusters).  We will stay in this paragraph on an informal level and just explain the main ideas, and leave it to the reader as an easy exercise to turn this into rigorous statements.
It is known that the measure on the outer boundary $\gamma$ of a Brownian loop $\theta$ is a  SLE$_{8/3}$ loop measure \cite{MR2350053}. Therefore, conditionally on an appropriately discovered part $\partial$ of $\gamma$, the rest of $\gamma$ is distributed like a chordal SLE$_{8/3}$ in the complement of $\partial$, which is known to satisfy chordal restriction property with exponent $5/8$. In fact, 
$\theta$ will  satisfy a chordal interior restriction in the complement of $\partial$, due to the restriction property of the Brownian loop itself and some similar arguments as in Proposition \ref{prop2.5}.

The trace of $\theta$ is the union of a number of  excursions of the Brownian loop away from $\partial$. However, it is not distributed like a Poisson point process of Brownian excursions in the complement of $\partial$, as we now explain: For some Brownian loop $\theta$ encircling the origin,
conditionally on a part $\partial$ of the outer boundary of $\theta$, with positive probability, $\partial$ contains cut-points of $\theta$ with respect to the origin (see Figure \ref{fig:loop-cut}) i.e. points that separate the unbounded connected component of the complement of the loop from the connected component  that contains the origin (see \cite{MR1009442}). 

However, in a Poisson point process $\Lambda^{\text{BE}}$ of intensity $5/8$ (here this is with respect to the normalization so that the obtained collection satisfies chordal restriction with exponent $5/8$) of excursions away from the right side of $\partial$ in the complement of $\partial$ (see Figure \ref{fig:loop-cut-excursions}) conditioned to encircle the origin, there is a.s. no cut-point on $\partial$. This is because in $\Lambda^{\text{BE}}$, after removing  the big excursions encircling the origin, the union of all the other excursions is absolutely continuous w.r.t. $\Lambda^{\text{BE}}$ and do not have cut-points on $\partial$ (this corresponds to the fact the SLE$_{8/3}$ is a simple curve). 

\begin{figure}[h!]
\centering
\begin{subfigure}[t]{0.48\textwidth}
\centering
\includegraphics[width=0.43\textwidth]{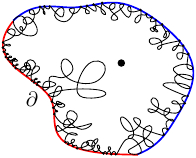}
\caption{The Brownian loop encircling the origin has positive probability to have local cut-points on $\partial$.}
\label{fig:loop-cut}
\end{subfigure}
\quad
\begin{subfigure}[t]{0.48\textwidth}
\centering
\includegraphics[width=0.43\textwidth]{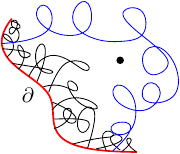}
\caption{The union of the Poisson point process of excursions with $\partial$ a.s. has no cut-points on $\partial$.}
\label{fig:loop-cut-excursions}
\end{subfigure}
\caption{Brownian loop and Poisson point process of excursions.}
\end{figure}

\section{The one-point pinned configuration}\label{sec:one-point}

Recall that the paper \cite {MR2979861} consists of two almost independent parts: 
\begin {itemize}
 \item On the one hand, it shows that the random collection of disjoint loops that satisfy a certain Markovian property can be explored progressively in a certain way, which in turn implies that they can be defined in terms of branching SLE trees. This gives a Markovian characterization of CLEs. 
 \item On the other hand, the outermost cluster boundaries of subcritical and critical Brownian loop-soup do satisfy this Markovian property. This shows the existence of the CLEs.   
\end {itemize}
One idea in the proofs in the coming sections will be to follow the exploration procedure (and the conditional distribution of the remaining to be explored configurations) of a CLE that has been defined via 
loop-soup clusters, and to study along the way the conditional distribution of the Brownian loop-soup itself (not just of the CLE) given the already explored pieces.  

In this section, we first describe the Markovian CLE-exploration on the loop-soup introduced in \cite {MR2979861} in order to define the one-point pinned configuration, and check that the decomposition ideas can be pushed through. Then we make two different decompositions of the one-point pinned configuration and thus derive the conformal restriction property of some specific subset of the one-point pinned configuration.

\subsection{Markovian exploration and the one-point pinned configuration}\label{sec:2.1}
In this section, we perform on the loop-soup the Markovian CLE-exploration which is described in detail in \cite{MR2979861}. 
The exploration process only depends on the CLE associated to the loop-soup. However we keep track of the additional loop structure inside each complete cluster.
For instance, one needs to show the convergence-type results for the metric (that we will define later) on the space of countable collections of loops.
In the following, we will describe the successive steps of the exploration process and recall some CLE-related properties for the loop-soup without providing all the proofs. In the meantime, we will clarify the specificities of the loop-soup case.

Let $\Gamma$ be the loop-soup in $\U$ which is subject to our exploration.
At this stage, we only need to perform the simple version of exploration where we discover the entire clusters immediately without tracing along their boundaries.
 Let $\theta_0$ be the complete cluster in $\Gamma$ which surrounds the origin.
We explore progressively the complete clusters in $\U$ that intersect the segment $[0,1]$ by moving from right to left. Set  $T:=\sup\{t\ge0, t\in\theta_0\}$ to be the first moment that we discover  $\theta_0$. Let $\tilde A$ be the filling  of the union of the segment $(T,1]$ with all the complete clusters that it intersects.
Let  $\Psi_0$ be the conformal map from $\U\setminus\tilde A$ onto $\U$ which sends the points $0, T$ to the points $0,1$. 
Here we  define the following probability measure on pinned complete clusters.
\begin{definition}\label{def:nui}
Let $\nu_0$ be the law of the complete cluster $\Psi_0(\theta_0)$.
\end{definition}

It is shown in \cite{MR2979861} that the map $\Psi_0$ can be approximated by the iteration of discrete small exploration maps. 
We refer the reader to \cite{MR2979861} for a complete description and proof. Here we make a brief summary of the main ideas: We fix some $\eps>0$. For each step, we discover all the complete clusters in $\Gamma$ that intersect $D(y,\eps)$, where the point $y$ is well chosen to approximate the to-be-explored part of the original segment $[0,1]$. As long as we have not discovered the complete cluster $\theta_0$, we map the connected component containing $0$ of the complement of the discovered complete clusters back to $\U$ by some conformal map $\varphi$ such that $\varphi(0)=0, \varphi'(0)>0$. Let $\Psi_0^\eps$ be the conformal map generated by the above-mentioned $\eps$-discrete exploration, namely the composition of all the maps $\varphi$ of each step up until the stopping step, together with a rotation in the end so that it maps the $y$ of the ending step back to the point $1$.
Let us note the two facts:
\begin{itemize}
\item
Each step is i.i.d. modulo the choice of $y$. Hence the law of $\Psi_0^\eps(\theta_0)$ is the same as the law of $\theta_0$ conditioned on $\{\theta_0\cap D(1,\eps)\not=\emptyset\}$.
\item
The conformal maps $\Psi_0^\eps$ converge a.s. to the map $\Psi_0$ as $\eps\to 0$.
Here we say a sequence of functions $(\Psi_n)$ converges to $\Psi$ in the sense that for all proper compact subsets $K$ of $\U$, the functions $\Psi_n^{-1}$ converge uniformly to $\Psi^{-1}$ in $K$.
\end{itemize}
As a consequence, the authors showed in \cite{MR2979861} the Proposition 4.1 that we restate below using our notations.
Here $\gamma_0$ denote the outer boundary of $\theta_0$.
\begin{proposition}[Proposition 4.1 in \cite{MR2979861}]\label{prop4.1}
As $\eps\to 0$, the law of $\gamma_0$ conditioned on the event $\{\gamma_0\cap D(1,\eps)=\emptyset\}$ converges to the law of the outer boundary of $\Psi_0(\theta_0)$ (using for instance the weak convergence with respect to the Hausdorff topology on compact sets).
\end{proposition}

Our goal is to derive an analogue of this proposition for the Brownian loop-soup, namely to show the convergence for the entire complete cluster but not only for the outer boundary of the complete cluster. For this purpose,
 we need to choose a metric on the space of countable collections of loops. 
For two finite collections of loops $\mathcal C$ and $\mathcal C'$, let us define
\begin{align*}
d_M=\min_\sigma \max_{\gamma\in\mathcal C} d(\gamma,\sigma(\gamma)) \wedge M
\end{align*}
where $d$ is the Hausdorff distance between the traces of two loops and $\min$ is taken over all bijections $\sigma$ from $\mathcal C$ to $\mathcal C'$ with the convention that $\min\emptyset=\infty$ (so that the distance $d_M$ between two collections of loops  is at most $M$).
In a Brownian loop-soup $\Gamma$ in $\U$, for $n\ge 1$, the collection $\Gamma_n$ of loops of diameters between $2^{-n}$ and $2^{-n-1}$ is a.s. finite. Let $\Gamma_0$ be the finite collection of loops of size bigger than $1/2$. We can define the distance between two loop-soups $\Gamma$ and $\Gamma'$ by
\begin{align*}
d^*(\Gamma,\Gamma')=\sum_{n=0}^\infty 2^{-n} d_M(\Gamma_n,\Gamma'_n).
\end{align*}
Under this distance $d^*$, the convergence of the conformal maps $\Psi^{\eps}_0$ as described earlier implies the convergence of $\Psi^{\eps}_0(\theta_0)$ (as a collection of loops). This is because for any fixed $\delta>0$, the collection of loops in $\Gamma$ with diameter greater than $\delta$ is finite and contained in some proper compact subset of $\U$. 
This leads to the following lemma.
\begin{lemma}\label{lem:def1}
The law of $\theta_0$ conditioned on the event $\{\theta_0\cap D(1,\eps)\not=\emptyset\}$ converges as $\eps\to 0$ to $\nu_0$.
\end{lemma}

From here on, we switch to the upper half-plane setting by applying  the conformal map $\varphi: z\mapsto i(1+z)/(1-z)$, because it is convenient to have the scaling invariance.  The metric on the collections of loops in $\H$ can be taken as the image under $\varphi$ of the metric $d^*$. 
For a loop-soup $\Gamma$ in $\H$, we denote by $\theta(i)$ the complete cluster that surrounds the point $i$. We define  the conformal map $\Psi^i$ in a similar way as $\Psi_0$, but in the $\H$ setting, see Figure \ref{fig:exploration-h}. 
Definition \ref{def:nui} and Lemma \ref{lem:def1} have the following counterparts.

\begin{definition}\label{def:3.4}
Let $\nu(i)$ be the law of the one-point pinned complete cluster $\Psi^i(\theta(i))$.
\end{definition}

\begin{lemma}\label{lem:3.4}
The law of $\theta(i)$ conditioned on the event $\{\theta(i)\cap D(0,\eps)\not=\emptyset\}$ converges as $\eps\to 0$ to $\nu(i)$.
\end{lemma}
\begin{figure}[h!]
\centering
\includegraphics[width=\textwidth]{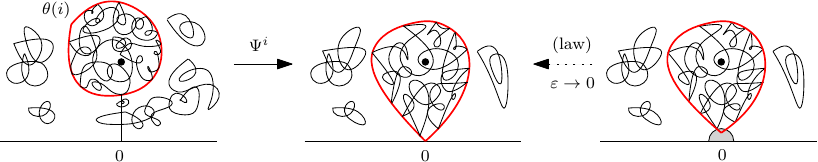}
\caption{We explore the complete cluster $\theta(i)$ which surrounds the point $i$ by moving upwards the segment $[0,i]$. }
\label{fig:exploration-h}
\end{figure}

Let us recall Lemma \ref{lem:Q-W} and the definition of the law $P_0$. Note that the event $\{\theta(i)\cap D(0,\eps)\not=\emptyset\}$ depends only on the outer boundary $\gamma(i)$ of $\theta(i)$. 
Let $\mu(i)$ denote the law of the outer boundary of $\Psi^i(\theta(i))$, which is simply the one-point pinned loop measure for CLE. 
Recall the map $\Phi$ defined in \S\,\ref{sec:notations}.
We are now ready to state the following lemma, which is a direct consequence of Lemma \ref{lem:Q-W}.
\begin{lemma}\label{lem:ind}
The measure $\nu(i)$ is the image of the product measure $\mu(i)\otimes P_0$ under the map $\Phi$.
\end{lemma}

In fact, Definition \ref{def:3.4}, Lemma \ref{lem:3.4} and Lemma \ref{lem:ind} can be seen as three different approaches to define the same measure $\nu(i)$.  

Now we want to extend the probability measure $\nu(i)$ to an infinite measure $\nu$ on pinned complete clusters in $\H$ that do not depend on the point $i$. 
%
We recall that in  \cite{MR2979861}, it is shown that the probability measure $\mu(i)$ can be extended to an infinite measure $\mu$ on pinned loops which satisfies the properties listed below. 
Let $u(\eps)$ be the probability that the loop $\gamma(i)$ intersects $D(0,\eps)$. We denote by $\mu(z)$ the measure $\mu$ restricted to loops that surround $z$.
\begin{itemize}
\item For all $z\in\H$, the measure $\mu(z)$ is the limit of $u(\eps)^{-1}$ times the law of $\gamma(z)$ in a CLE, restricted to the event $\{\gamma(z)\cap D(0,\eps)\not=\emptyset\}$.
\item For any conformal transformation $\psi$ from $\H$ onto itself with $\psi(0)=0$, we have
\begin{align*}
\psi\circ\mu=|\psi'(0)|^{-\beta}\mu.
\end{align*}
This is called the \emph{conformal covariance property} of $\mu$.
\item For each $z\in\H$, the mass of $\mu(z)$ is finite and equal to $\psi'(0)^\beta$, where $\psi$ is the conformal map from $\H$ onto itself with $\psi(0)=0$ and $\psi(z)=i$.
\item For two different points $z,z'\in\H$, the measures $\mu(z)$ and $\mu({z'})$ coincide on the loops that surround both $z$ and $z'$.
\item As $\eps\to 0$, $u(\eps)^{-1}$ times the law of the largest loop in a CLE that touches $D(0,\eps)$ converges vaguely to $\mu$.
When we say `largest', we will always mean that it has the largest diameter. This notion is not conformally invariant.
\end{itemize}
Note that the description of $\mu(z)$ for all $z$ fully determines the measure $\mu$, because every pinned loop necessarily surrounds a small disc. 
Now we can define $\nu$ using the measure $\mu$. Recall the definition of $\Phi$ in \S\,\ref{sec:notations}.
\begin{definition}\label{def:nu}
Let $\nu$ be the image under $\Phi$ of $\mu\otimes P_0$.
\end{definition}

The following lemma is a consequence of  the Definition above and of the properties of $\mu$.
\begin{lemma}\label{lem:nuz}
For a loop-soup in $\H$, $u(\eps)^{-1}$ times the probability measure of the complete cluster with largest diameter that intersects $D(0,\eps)$ converges vaguely as $\eps\to 0$ to $\nu$ (hence the measure converges weakly when restricted to the complete clusters that encircle a certain point $z$, or to the ones with a diameter greater than some $c>0$).
\end{lemma}

The measure $\nu$ also inherits the conformal covariance property from $\mu$.
\begin{lemma}[Conformal covariance]\label{lem:conf-cov}
For any conformal transformation $\psi$ from $\H$ onto itself with $\psi(0)=0$, we have
\begin{align*}
\psi\circ\nu=|\psi'(0)|^{-\beta}\nu.
\end{align*}
As a consequence, for each $z\in\H$, the mass of $\nu(z)$ is finite and equal to $\psi'(0)^\beta$.
\end{lemma}

Now we can upgrade $\nu$ to a measure on {one-point pinned configurations}  by adding back all the loops outside the pinned complete cluster. Recall  that the notion of one-point pinned configurations and  the function $\Phi^{\text{up}}$  are defined in \S\,\ref{sec:notations}.
\begin{definition}\label{def:nu-bar}
Let the one-point pinned configuration measure $\bar\nu$ be the image of the product measure $\nu\otimes \P_\H$ under the map $\Phi^{\text{up}}$.
\end{definition}

The measure $\bar\nu$ is of course conformally covariant, due to the conformal covariance of $\nu$ and $\P_\H$. 
The measure $\bar\nu$ is  also naturally related to the exploration process. 
Note that $\bar\nu$ restricted to the event that the pinned complete cluster encircles the point $i$ is a probability measure. It is equal to the probability measure of $\Psi^i(\Gamma')$, where $\Psi^i$ is the conformal map defined in Figure \ref{fig:exploration-h} and $\Gamma'$ is the collection of loops in $\Gamma$ that stay entirely in the domain of definition of $\Psi^i$.

The results we get until here are not surprising since they are inherited from the CLE. 
In the next section, we are going to exploit further the loop-soup structure and break apart the loops in the same complete cluster.

\subsection{Decomposition of the one-point pinned configuration}

In this section, we will make two decompositions of the one-point pinned configuration and then obtain the conformal restriction property of the measure $\bar\rho_A$, which is a measure  on a certain subset of loops of the one-point pinned configuration defined via the first decomposition.

\subsubsection{The first decomposition}

Let $\Gamma$ be a Brownian loop-soup in $\H$. Recall that $\mathcal{A}$ is the set of all relatively closed sets $A\subset\H$ such that $d(0,\H\setminus A)>0$. 
For  $A\in\mathcal{A}$,  $\Gamma$ can be decomposed as the union of two independent collections of loops:  the collection $\Gamma^A$ of loops that intersect $A$ and the collection $\Gamma_{\H\setminus A}$ of loops that stay entirely in $\H\setminus A$.

We want to make the same kind of decomposition for the measure $\bar\nu$ on pinned configurations $(\theta,\Gamma_\theta)$ in $\H$. Note that a pinned configuration $\Lambda=\theta\cup\Gamma_\theta$ is also the union of the two collections of loops:
 the collection $\Lambda^A$ of loops that intersect $A$ and the collection $\Lambda_{\H\setminus A}$ of loops that stay entirely in $\H\setminus A$. 
We want to show that these two collections are `independent' and that $\Lambda^A$ is supported on pinned configurations and $\Lambda_{\H\setminus A}$ is just a loop-soup in $\H\setminus A$ (hence is equal in law to $\Gamma_{\H\setminus A}$).
Since we are dealing with infinite measures,  the `independence' needs to be stated in terms of product measures.

As the first step, we show that the loops in $\Lambda_{\H\setminus A}$ are not needed to pin the complete cluster.
\begin{lemma}\label{lem:theta-a-pin}
For a.e. pinned complete cluster $\theta$ in the support of $\nu$, the collection $\theta^A$ of loops in $\theta$ that intersect $A$ contains a  pinned complete cluster.
\end{lemma}
\begin{proof}
We only need to prove the statement for the case where $\theta$ is not entirely contained in $A$. We need to rule out the case where there is a sequence of complete clusters $\theta_n$ in $\theta^A$ such that they all intersect one common complete cluster in $\theta_{\H\setminus A}$ and $d(\theta_n,0)$ goes to $0$.
In this case, the complete clusters $\theta_n$ all have diameter greater or equal than $d(0,\H\setminus A)$ and the outer boundary of $\theta$ is therefore discontinuous at $0$, which contradicts the fact that the outer boundary of $\theta$ has the law of a pinned SLE loop (see  \cite{MR2979861}).
\end{proof}
Let $V$ be the map that maps a pinned complete cluster $\theta$ to $V(\theta)$ which is the unique pinned complete cluster in the collection of loops $\theta^A$. The existence of $V(\theta)$ follows from Lemma \ref{lem:theta-a-pin} and the uniqueness is due to the fact that two different clusters in a same loop-soup a.s. have disjoint closures hence cannot be both pinned at the origin.
 Let us define the measures $\rho_A$ and $\rho_A^z$.

\begin{definition}
Let $\rho_A$ be the image of the measure $\nu$ under the map $V$.
For all $z\in A$, let $\rho_A^z$ be the measure $\rho_A$ restricted to the complete clusters that encircle the point $z$, which is a finite measure.
\end{definition}

Let us choose some $\eps$ such that $0<\eps <d(0,\H\setminus A)$. 
Let $\theta^\eps$ be the largest  (in the sequel, by `largest' we mean that it has the largest diameter) complete cluster in $\Gamma$ that intersects $D(0,\eps)$ and we denote its law by $P\theta^\eps$.
Let $\pi^\eps_A$ be the largest complete cluster in $\Gamma^A$ that intersects $D(0,\eps)$ and we denote its law by $P\pi^\eps_A$. We show that $\rho_A^z$ is equal to the following limit.

\begin{lemma}\label{lem1}
For all $z\in A$, the measure $u(\eps)^{-1}P\pi_A^\eps \mathbf{1}_{z\in\pi^\eps_A}$ converges as $\eps\to 0$ to  $\rho_A^z$.
\end{lemma}
\begin{remark}
For all pinned complete clusters $\theta$, by Lemma \ref{lem:theta-a-pin}, $V(\theta)$ necessarily surrounds a small disk in $A$. Therefore, the description of $\rho_A^z$ for all $z\in A$ fully determines the measure $\rho_A$. We can also say that $u(\eps)^{-1}P\pi_A^\eps$ converges vaguely to  $\rho_A$.
\end{remark}

\begin{figure}[h!]
    \centering
    \begin{subfigure}[t]{0.43\textwidth}
    \centering
        \includegraphics[width=.9\textwidth]{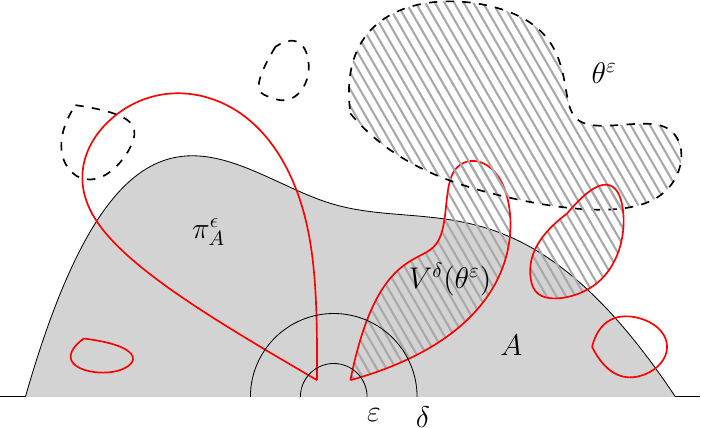}
    \end{subfigure}
    \qquad\quad    
     \begin{subfigure}[t]{0.43\textwidth}
     \centering
        \includegraphics[width=.9\textwidth]{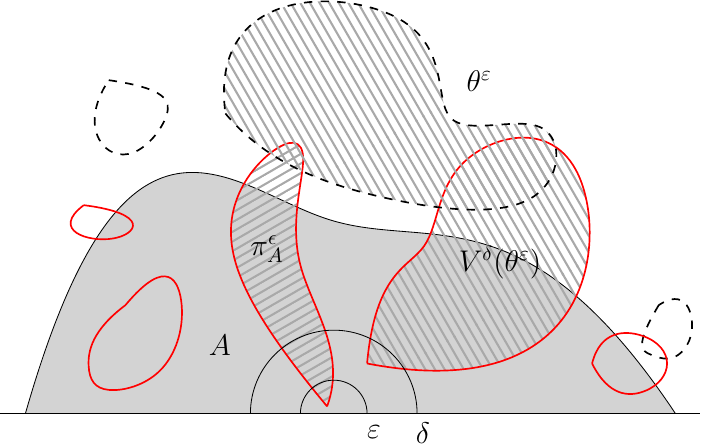}
    \end{subfigure}
    \caption{Sketch of different collections of loops in a loop-soup $\Gamma$. The plain loops represent the complete clusters in $\Gamma^A$ and the dotted loops represent the complete clusters in $\Gamma_{\H\setminus A}$. Their union gives rise to complete clusters in $\Gamma$, such as the $\theta^\eps$ (shaded). We represent here two (rare) configurations where $J^{\eps,\delta}$ do not occur.}
    \label{fig:pi}
\end{figure}

\begin{proof}
We know by Lemma \ref{lem:nuz} that $u(\eps)^{-1}$ times the measure of $\theta^\eps$ converges vaguely to the measure $\nu$ as $\eps\to 0$. The idea is to consider the respective images  of $P\theta^\eps$ and $\nu$ under the  map $V$ to deduce the convergence in the present lemma. However the direct application of $V$ to $\theta^\eps$ yields an empty set. Therefore we first define  the map $V^\delta$ that maps a collection of loops $\theta$ to another collection of loops $V^\delta(\theta)$ which is  the largest complete cluster in $\theta^A$ that intersects $D(0,\delta)$, see Figure \ref{fig:pi}. 

Let $J^{\eps,\delta}$ be the event that $V^\delta(\theta^\eps)$ is  equal to $\pi_A^\eps$. We first try to prove (\ref{j}): Conditionally on the event that $\pi_A^\eps$ encircles some point $z\in A$, the probability of $J^{\eps,\delta}$ is close to $1$. In fact, conditionally on the event that $\pi_A^\eps$ encircles  $z$, the loops in $\Gamma^A \setminus \pi_A^\eps$ are distributed like an independent loop-soup in the unbounded connected component of $\H\setminus\pi_A^\eps$, restricted to the loops that intersect $A$.  If we add back to $\Gamma^A \setminus \pi_A^\eps$ independently in a Poissonnian way all the loops that either do not intersect $A$ or intersect  the filling of $\pi_A^\eps$, then we get a loop-soup in $\H$, whose complete clusters are all bigger than those of $\Gamma^A \setminus \pi_A^\eps$. Among these enlarged complete clusters, the probability that there exist one complete cluster that intersect $D(0,\delta)$ with diameter greater than $d(0, \H\setminus A)$ is $o_\delta(1)$. Therefore with probability $1-o_\delta(1)$, the complete 
cluster $V^\delta(\theta^\eps)$ is either equal to $\pi_A^\eps$, or is contained in $A$. However, if $V^\delta(\theta^\eps)$ is contained in $A$, then $\theta^\eps$ must be equal to $V^\delta(\theta^\eps)$, in which case it must also be equal to $\pi_A^\eps$. Therefore we have proved  
\begin{align}\label{j}
\P\left[J^{\eps,\delta} | z\in\pi_A^\eps\right]=1-o_\delta(1).
\end{align}
Therefore for all continuous and bounded functions $Y$, we have the following.
\begin{align}
\label{l1}
u(\eps)^{-1} \E \left[ Y(\pi_A^\eps) \mathbf{1}_{z\in \pi_A^\eps} \right]&=u(\eps)^{-1}\E\left[  Y(\pi_A^\eps) \mathbf{1}_{\{z\in \pi_A^\eps\}\cap J^{\eps,\delta}}  \right]+o_\delta(1)\\
\label{l2}
&=u(\eps)^{-1}\E\left[  Y(V^\delta(\theta^\eps)) \mathbf{1}_{\{z\in V^\delta(\theta^\eps)\}\cap J^{\eps,\delta}}  \right]+o_\delta(1)\\
\label{l3}
&=u(\eps)^{-1}\E\left[  Y(V^\delta(\theta^\eps)) \mathbf{1}_{\{z\in V^\delta(\theta^\eps)\}}  \right]+o_\delta(1).
\end{align}
The line (\ref{l1}) is due to (\ref{j}), the fact that $\P( z\in\pi_A^\eps)/ u(\eps)\le 1$ and that
$$u(\eps)^{-1} \P((J^{\eps,\delta})^c, z\in\pi_A^\eps)=\P((J^{\eps,\delta})^c | z\in\pi_A^\eps)\, \P( z\in\pi_A^\eps)/ u(\eps)=o_\delta(1).$$
The line (\ref{l3}) can be deduced similarly using the fact 
$\P\left[J^{\eps,\delta} | z\in V^\delta(\theta^\eps) \right]=1-o_\delta(1)$, which can be deduced similarly to how we deduce (\ref{j}).

Now we first let $\eps\to 0$ in (\ref{l3}) and try to show that  it converges to
\begin{align}\label{l4}
\nu\left[ Y(V^\delta(\theta)) \mathbf{1}_{z\in V^\delta(\theta)} \right]+o_\delta(1).
\end{align}
Note that Lemma \ref{lem:nuz} implies that  $u(\eps)^{-1} P\theta^\eps  \mathbf{1}_{\{z\in \theta^\eps\}} $ converges weakly to $\nu  \mathbf{1}_{\{z\in \theta\}}$.  
However, we cannot directly apply the weak convergence, because $V^\delta$ is in general discontinuous at $\theta$. 
The way we have defined $d^*$ makes it that we can add as many small loops to $\theta$ as we want to create a collection $\theta'$ which is at small distance from $\theta$ but has a very different cluster structure (disjoint loops in $\theta$ can be connected in $\theta'$ via a long chain of small loops).
Nevertheless, we can surpass this problem by constructing a sequence $\theta^{\eps_n}$ of complete clusters (whose laws are $P\theta^{\eps_n}$ restricted to those $\theta^{\eps_n}$s that encircle $z$ and then renormalized to total mass $1$) that converges a.s. to $\theta$. More precisely, let $\gamma_n$ be a simple loop that has the same law as the CLE loop in $\H$ encircling $z$ conditioned to intersect $D(0,\eps_n)$. By the results in \cite{MR2979861} that we have recalled under Lemma \ref{lem:ind}, we know that the laws of $\gamma_n$ converge weakly to $\mu(z)$ renormalized to total mass $1$. By Skorokhod's representation theorem, we can find a sequence $\gamma_n$ that converges a.s. to $\gamma$. Let $\xi$ be a collection of loops with the law $P_0$.  
Let $\theta^{\eps_n}$ be equal to $\Phi(\gamma_n,\xi)$. By definition \ref{def:nu} and Lemma \ref{lem:nuz}, the $\theta^{\eps_n}$s have the desired distributions and converge a.s. to $\theta$ which is equal to $\Phi(\gamma,\xi)$.
The maps $\Phi(\gamma_n,\cdot)$ and $\Phi(\gamma,\cdot)$ do not change the cluster structure of $\xi$ (two loops in $\xi$ touch each other if and only if their images in $\theta^{\eps_n}$ or $\theta$ touch each other). Moreover, if a loop in $\theta$ touches $A$, then it must touch the interior of $A$.
Then, knowing that $\gamma_n$ converges a.s. to $\gamma$, we can conclude that $V^\delta(\theta^{\eps_n})$ also converges a.s. to $V^\delta(\theta)$.
Since $Y$ is continuous, we get the desired convergence from (\ref{l3}) to (\ref{l4}).

Then we let $\delta\to 0$ and it converges to $\nu\left[ Y(V(\theta)) \mathbf{1}_{z\in V(\theta)} \right]$.
Therefore we have proved
\begin{align*}
u(\eps)^{-1} \E \left[ Y(\pi_A^\eps) \mathbf{1}_{z\in \pi_A^\eps} \right] \underset{\eps\to 0}\longrightarrow \nu\left[ Y(V(\theta)) \mathbf{1}_{z\in V(\theta)} \right],
\end{align*}
thus proving the lemma.
\end{proof}

Once we have defined $\rho_A$, we  can upgrade it to a measure on one-point pinned configurations. 
Recall the map $\Phi^{\text{up}}$ and the probability measure $\P^A$ defined in \S\,\ref{sec:notations}.

\begin{definition}\label{def:rho-bar}
Let $\bar\rho_A$ (resp. $\bar\rho^z_A$) be the image of the product measure $\rho_A\otimes \P^A$ (resp. $\rho^z_A\otimes \P^A$) under the map $\Phi^{\text{up}}$. 
\end{definition}

Now let us define a  map $U$, which will turn out to relate  $\bar\rho_A$ to the measure $\bar\nu$ in Definition \ref{def:nu-bar}.
Let $(\theta,\Gamma_\theta)$ be a loop configuration in $\H$. Let $\Gamma'$ be another collection of loops in $\H$.  We define the map $U$ that maps $((\theta,\Gamma_\theta),\Gamma')$ to $(\theta', \Gamma'')$ where $\theta'$ is the complete cluster containing $\theta$ in the new loop collection $\tilde\Gamma=\theta\cup\Gamma_\theta\cup\Gamma'$, and where $\Gamma''$ is the collection of all the loops in $\tilde\Gamma$ which are not in $\theta'$. We also denote by $U^1$ the map from $((\theta,\Gamma_\theta),\Gamma')$ to $\theta'$ (see Figure \ref{fig12}).

\begin{figure}[h!]
\centering
\includegraphics[width=0.45\textwidth]{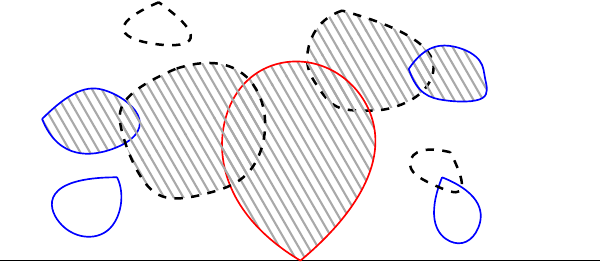}
\caption{We only draw the outer boundaries of the complete clusters. The complete cluster $\theta$ is represented by the red loop. The complete clusters in $\Gamma_\theta$ are represented by the blue loops. The  complete clusters in $\Gamma'$ are represented by the dashed loops. The shaded part represents the complete cluster $\theta'$ and all the loops in $\tilde\Gamma$ that are not in $\theta'$ belong to $\Gamma''$.}
\label{fig12}
\end{figure}

The relation between $\bar\rho_A$ and $\bar\nu$ is given by the following lemma.

\begin{lemma}[First decomposition]\label{lem2}
The measure $\bar\nu$ is the image of $\bar\rho_A\otimes \P_{\H\setminus A}$ by the map $U$.
\end{lemma}

\begin{proof}
We define $\tilde\theta^\eps$ to be the complete cluster  in $\Gamma$ which contains the collection $\pi_A^\eps$. 
In other words, $\tilde\theta^\eps$ is the image of $((\pi_A^\eps, \Gamma^A\setminus \pi_A^\eps),\Gamma_{\H\setminus A})$ under the  map $U^1$. The measure $u(\eps)^{-1}P \pi_A^\eps$ converges vaguely to $\rho_A$, hence $u(\eps)^{-1}$ times the measure of $(\pi_A^\eps, \Gamma^A\setminus \pi_A^\eps)$ converges vaguely to $\bar\rho_A$.
Now we want to argue that $u(\eps)^{-1}P\tilde\theta^\eps$ converges vaguely to $\nu$.
We already know that $u(\eps)^{-1}P\theta^\eps$ converges vaguely to $\nu$. The event $\{\tilde\theta^\eps=\theta^\eps\}$ does not always happen (as in Figure \ref{fig:pi}), but it happens with high probability conditionally on $\{z\in\pi_A^\eps\}$.
We can adapt the same type of reasoning as in the proof of Lemma \ref{lem1} to conclude that $u(\eps)^{-1}P\tilde\theta^\eps$ converges indeed vaguely to $\nu$.
This implies that the measure $\nu$ is the image of $\bar\rho_A\otimes \P_{\H\setminus A}$ under the map $U^1$, which further implies the lemma.
\end{proof}

\subsubsection{The second decomposition}

In this section, we decompose the pinned configuration differently. 
For $A\in\mathcal{A}$, a loop-soup $\Gamma$ is the union of two independent collections of loops: $\Gamma_A$ consisting of all the loops in $\Gamma$ that stay entirely in $A$ and of $\Gamma^{\H\setminus A}$ consisting of all the loops in $\Gamma$ that do not stay entirely in $A$. 

Similarly,  a pinned configuration $\Lambda$ is the union of the two collections of loops:
 the collection of loops $\Lambda_A$ that stay entirely in $A$ and the collection of loops $\Lambda^{\H\setminus A}$ that do not stay entirely in $A$. 
We want to show that $\Lambda_A$ is supported on pinned configurations and $\Lambda^{\H\setminus A}$ is just an `independent' loop-soup.
Of course, the `independence'  is also in the sense of product measures.

Similarly to the first decomposition, we need to show that the loops in $\Lambda^{\H\setminus A}$ are not needed to pin the complete cluster. 
This requires a bit more arguments than proving Lemma \ref{lem:theta-a-pin}, because a priori the loops in $\Lambda^{\H\setminus A}$ can contain some cluster that touches the origin. However we will show that it is a.s. not the case. 

\begin{lemma}
Let $\Lambda^i$ be a pinned complete cluster with the law $\nu(i)$, then almost surely the origin does not belong to any of the loops in $\Lambda^i$.
\end{lemma}

\begin{figure}[h!]
\centering
\includegraphics[width=0.36\textwidth]{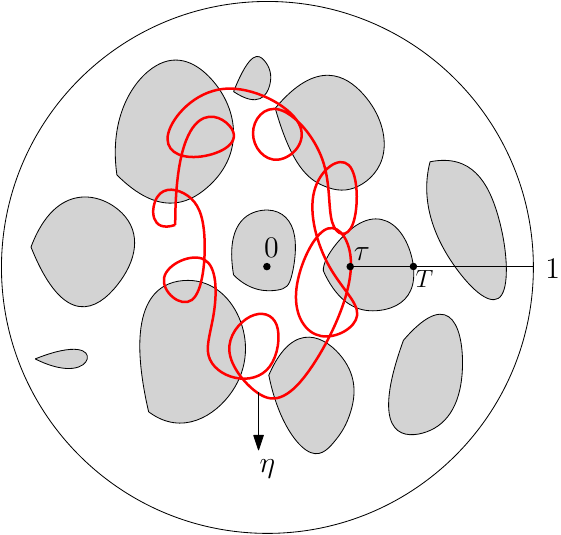}
\caption{The loop-soup $\Gamma$ (we only draw  its filled complete clusters, in grey), the loop $\eta$ and the stopping times $\tau, T$.}
\label{fig:explore-pinned}
\end{figure}

\begin{proof}
It is equivalent, but more convenient for us to work in $\U$, and to prove that the point $1$ a.s. does not belong to the pinned cluster $\Psi_0(\theta_0)$ which is obtained from the exploration process (see Definition \ref{def:nui}).

Let us first consider a loop-soup $\Gamma$ in $\U$, and an independent Brownian loop $\eta$ which is chosen with the renormalized Brownian loop measure restricted to those loops in $\U$ with diameter greater than $\delta$.
In $\Gamma\cup\{\eta\}$, there is a complete cluster  that encircles the origin that we denote by $\tilde\theta_0$.
When we explore along the segment $[0,1]$ from the right to the left, let $T$ be the first time that we encounter $\tilde\theta_0$ (see Figure \ref{fig:explore-pinned}).
We want to show that the point $T$ is a.s. not on $\eta$.  In the case where $\eta\cap [0,1]=\emptyset$, the point $T$ is obviously not on $\eta$. Otherwise, let $\tau$ be the first time that we encounter the loop $\eta$ (by going from right to left along $[0,1]$). 
Note that the point $\tau$ is independent of the loop-soup $\Gamma$, hence $\tau$ is a.s. in the interior of some loop in $\Gamma$. 
 Therefore, $\tau$ must also be in the interior of $\tilde \theta_0$ and hence we must have $\tau<T$, thus the point $T$ is not on $\eta$ either.

Now let us consider the loop-soup $\Gamma$ alone. 
 Let $\eta$ be a loop uniformly chosen within the loops in $\Gamma$ with diameter greater than $\delta$. Then $\eta$ is  distributed like a Brownian loop in $\U$ with diameter greater than $\delta$, and is independent of the rest of the loop-soup $\Gamma\setminus\{\eta\}$. The law of $\Gamma\setminus\{\eta\}$ is absolutely continuous with respect to that of $\Gamma$. 
 Let $\theta_0$ be the complete cluster in $\Gamma$ that surrounds the origin and let $T$ be the right-most point on $[0,1]$ which is in the closure of $\theta_0$. 
We know that the point $T$ is a.s. not on $\eta$ by the argument of the paragraph above. 
Since this is true for any $\eta$ in $\Gamma$ with diameter greater than $\delta$ and for any $\delta$, we conclude that the point $T$ is a.s. not on any of the loops in $\Gamma$.
Consequently the point $1$ is not on any of the loops in $\Psi_0(\theta_0)$. 
\end{proof}

This implies in particular the following lemma.
\begin{lemma}\label{lem:extW}
After removing any finite number of loops from the collection $\Lambda^i$, among the complete clusters formed by the remaining loops, there exist a.s. a complete cluster pinned at $0$.
\end{lemma}
\begin{proof}
Let $R$ be the finite collection of loops that we have removed from $\Lambda^i$.
Since $0$ is not on any of the loops in $\Lambda^i$, we  only need to rule out the case where in the collection of loops $\Lambda^i\setminus R$, there is a sequence of complete clusters $\theta_n$ such that they intersect a common complete cluster of $R$ and $d(\theta_n, 0)$ goes to $0$. However this is impossible for similar reasons as in
Lemma \ref{lem:theta-a-pin}. Note that the distance from the loops in $R$ to $0$ is strictly positive.
\end{proof}

This lemma  tells us that a pinned complete cluster is  pinned by the small loops in an infinitesimal neighborhood of the origin. 
Now we are going to list the required lemmas for the second decomposition.
The main idea is analogous to the first decomposition, hence  we will omit similar  proofs and focus on the differences.
Let $W$ be the map that maps a pinned complete cluster $\theta$ to the image $W(\theta)$ which is the unique pinned complete cluster in  $\theta_A$. 
The existence of $W(\theta)$ follows from Lemma \ref{lem:extW} and the uniqueness is due to the fact that two different clusters in a same loop-soup a.s. have disjoint closures hence cannot be both pinned at the origin.
We now define the measure $\nu_A$.
\begin{definition}
Let $\nu_A$ be the image of the measure $\nu$ under the map $W$. For all $z\in A$, let $\nu_A^z$ be the measure $\nu_A$ restricted to the complete clusters that encircle the point $z$, which is a finite measure.
\end{definition}

Let $\theta_A^\eps$ be the largest complete cluster that intersects $D(0,\eps)$ formed by the loops in $\Gamma_A$. Then we have the following lemma.

\begin{lemma}\label{lem:theta-cvg}
For all $z\in A$, the measure $u(\eps)^{-1}P\theta_A^\eps \mathbf{1}_{z\in\theta^\eps_A}$ converges as $\eps\to 0$ to  $\nu_A^z$.
\end{lemma}
\begin{proof}
We define the map $W^\delta$ that maps a collection of loops $\theta$ to another collection of loops $W^\delta(\theta)$ which is the largest complete cluster that intersects $D(0,\delta)$ formed by the loops in $\theta$ that stay entirely in $A$.
Let $H^{\eps,\delta}$ be the event that $W^\delta(\theta^\eps)$ is equal to $\theta_A^\eps$.
Then conditionally on the event  $\{z\in\theta_A^\eps\}$ for some point $z\in A$, the probability of $H^{\eps,\delta}$ is close to $1$.
In fact, conditionally on the event that  $\theta_A^\eps$ encircles $z$, the loops in $\Gamma_A\setminus\theta_A^\eps$ are distributed like an independent loop-soup in  $A\setminus \theta_A^\eps$. 
On the event $(H^{\eps,\delta})^c$ for which  $W^\delta(\theta^\eps)$ is not equal to $\theta_A^\eps$, the union of the loops in $\Gamma^{\H\setminus A}$ and the loops in $\Gamma_A\setminus\theta_A^\eps$ must contain a cluster that touch both $D(0,\delta)$ and $\H\setminus A$. The probability of this event goes to $0$ as $\delta\to 0$. Therefore we have proved
\begin{align*}
\P\left[ H^{\eps,\delta} | z\in \theta_A^\eps \right]=1-o_\delta(1).
\end{align*}
The rest of the proof stays very similar to that of Lemma \ref{lem1}, hence we leave it to the reader.
\end{proof}

Now let $\theta_A(z)$ be the complete cluster in $\theta_A$ that encircles the point $z$.
We have the following lemma, which can be taken as an alternative definition of  $\nu_A^z$.
\begin{lemma}\label{lem:3.25}
For all $z\in A$, the measure $u(\eps)^{-1}P\theta_A(z) \mathbf{1}_{\theta_A(z)\cap D(0,\eps)\not=\emptyset}$ converges as $\eps\to 0$ to  $\nu_A^z$.
\end{lemma}
\begin{proof}
Note that, conditionally on the event ${\theta_A(z)\cap D(0,\eps)\not=\emptyset}$, the set $\theta_A(z)$ is with high probability equal to the set $\theta_A^\eps$. Then we can apply Lemma \ref{lem:theta-cvg} to conclude.
\end{proof}

Before stating the second decomposition lemma, let us point out that the measure $\nu_A$ is related to the previously defined measure $\nu$ (which is in fact  $\nu_\H$). 
It is reasonable to expect that for simply connected $A$, $\nu_A$ is equal to the image of the measure $\nu_\H$ under some given conformal map from $\H$ onto $A$, up to a covariance constant. 
For more general (not simply connected) domains $A\in\mathcal{A}$, we also expect the conformal covariance to hold.
Lemma \ref{lem:2nd-decomp} (second decomposition) that will be proved later  can be seen as a restriction property of the measure $\nu$. With a little abuse of notation (talking about $\nu_{A_1}, \nu_{A_2}$ as if they were probability measures), we can say that for $A_1\subset A_2$, if $\theta_{A_2}$ is a pinned complete cluster in $A_2$ with the law $\nu_{A_2}$, then the pinned complete cluster formed by the loops in $\theta_{A_2}$ that stay in $A_1$ is distributed like $\nu_{A_1}.$  
Now we state the above-mentioned properties in the following lemmas.

\begin{lemma}
Let $A_1, A_2$ be two domains such that there is a conformal map $\varphi$ from $A_1$ onto $A_2$. The measure $\nu_{A_1}$ is equal to $\varphi'(0)^{-\beta}$ times the image of $\nu_{A_2}$ under $\varphi^{-1}$.
\end{lemma}
\begin{proof}
The loop-soup itself is conformally invariant. Therefore $\varphi(\theta_{A_1}(z))$ has the same law as $\theta_{A_2}(z')$ where $z'=\varphi(z)$.
The characterization of Lemma \ref{lem:3.25} and the fact that $u(c\eps)/u(\eps)\to c^{\beta}$ imply that 
$\varphi\circ\nu_{A_1}^z=\varphi'(0)^{-\beta}\nu_{A_2}^{z'}$. Since this is true for all $z,z'$, we have proved the lemma.
\end{proof}

We can further upgrade  $\nu_A$ to a measure $\bar\nu_A$ on pinned configurations in $A$.

\begin{definition}\label{def:nu-bar}
Let the measure $\bar\nu_A$ be equal to the image of the measure $\nu_A\otimes\P_A$ under the map $\Phi^{\text{up}}$.
\end{definition}
It is clear that $\bar\nu_A$ also satisfies the same conformal covariance property as $\nu_A$.
\begin{lemma}\label{lem3.25}
Let $A_1, A_2$ be two domains such that there is a conformal map $\varphi$ from $A_1$ onto $A_2$. The measure $\bar\nu_{A_1}$ is equal to $\varphi'(0)^{-\beta}$ times the image of $\bar\nu_{A_2}$ under $\varphi^{-1}$.
\end{lemma}
Now we are ready to state the second decomposition lemma. The proof is quite similar to that of Lemma \ref{lem2}. We leave the details to the reader as an exercise.

\begin{lemma}[Second decomposition]\label{lem:2nd-decomp}
The measure $\bar\nu$ is the image of $\bar\nu_A\otimes \P^{\H\setminus A}$ by the map $U$.
\end{lemma}

\subsubsection{The conformal restriction property of the measure $\overline{\rho}_A$}

Recall that $\bar\rho_A$ is the measure defined in Lemma \ref{lem1}.

\begin{lemma}\label{lem:restriction1}
Let $(\theta,\Gamma)$ be a configuration supported by the measure $\bar\rho_A$.
Let $B, \tilde B$ be relatively closed sets such that $\H\setminus B, \H\setminus\tilde B\in\mathcal{A}$, and that there is a conformal map $\varphi$ from $\H\setminus B$ onto $\H\setminus \tilde B$.
Then the measure $\varphi(\bar\rho_A)$ restricted to the event $\{(\theta\cup\Gamma) \cap B=\emptyset\}$ is equal to $\varphi'(0)^{-\beta}$ times the measure $\bar\rho_{\varphi(A)}$ restricted to the event $\{(\theta\cup\Gamma) \cap \tilde B=\emptyset\}$.
\end{lemma}

\begin{figure}[h!]
\centering
\includegraphics[width=0.6\textwidth]{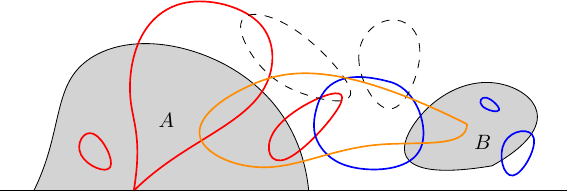}
\caption{The 4 different sets of loops: $C_1$ is drawn in red, $C_2$ in blue, $C_3$  in yellow and $C_4$  in dashed black.
We only draw the outer boundaries of the complete clusters composed of the loops in each of the collections $C_1,\cdots,C_4$.}
\label{fig:4sets}
\end{figure}

\begin{proof}
A one-point pinned configuration supported by $\bar\nu$ is the union of 4 disjoint collections of loops:
the collection $C_1$ of loops that touch $A$ but not $B$, the collection $C_2$ of loops that touch $B$ but not $A$, the collection $C_3$ of loops that touch both $A$ and $B$, and the collection $C_4$ of loops that touch neither $A$ nor $B$ (see Figure \ref{fig:4sets}).

By the first decomposition lemma applied to the set $A$, the collection $C_1\cup C_3$ follows the measure $\bar\rho_A$ and the collection $C_2\cup C_4$ is an independent loop-soup in $\H\setminus A$. 
The event $\{(C_1\cup C_3)\cap B=\emptyset\}$ is the same as the event that $C_3=\emptyset$.

By the second decomposition lemma applied to the set $\H\setminus B$, the collection $C_1\cup C_4$ follows the measure of $\bar\nu_{\H\setminus B}$ and is independent of the collection $C_2\cup C_3$. 

The two decompositions imply that the four sets $C_1,C_2,C_3,C_4$ are independent from each other.
In particular, the collection $C_1\cup C_3$ conditioned on the event $C_3=\emptyset$ is simply the collection $C_1$.

The same arguments apply for the pinned configuration in $\H\setminus \tilde B$ and we denote the corresponding sets by $\tilde C_1,\cdots, \tilde C_4$.
The lemma boils down to proving that the measure of $\varphi(C_1)$ is equal to $\varphi'(0)^{-\beta}$ times the measure of $\tilde C_1$.
However we know by Lemma \ref{lem3.25} that the measure of the pinned configuration $\varphi(C_1 \cup C_4)$ is equal to $\varphi'(0)^{-\beta}$ times the measure of $(\tilde C_1\cup \tilde C_4)$.
Note that $\varphi(C_1)$ is the image of $\varphi (C_1\cup C_4)$ under the map of keeping only those loops that intersect $A$. So is $\tilde C_1$ the image of $(\tilde C_1\cup \tilde C_4)$ under the same map. We can thus conclude.
\end{proof}

\section{The glued configuration}\label{sec:glued}
In this section, we will first define the two-point pinned configuration and the glued configuration.
Then we will make a decomposition of the glued configuration and prove the conformal restriction property of the boundary-touching loops in the glued configuration, thus proving the main theorem.

\subsection{Markovian exploration and the glued configuration}

In this section, we will continue to perform the Markovian exploration on the one-point pinned configuration in order to obtain at first the two-point pinned configuration and then the glued configuration.
This procedure is also explained for CLE in \cite{MR2979861}.
In fact,  the one-point pinned configuration satisfies the same CLE-type of restriction property as the original loop-soup and we can again explore it from the boundary. 
In the following, we describe the exploration process, but omit the proofs that can either be found in  \cite{MR2979861},  or are similar to those in \S\,\ref{sec:2.1}.

Note that the  $\bar\nu$-mass of the pinned configurations where the pinned complete cluster intersects the vertical half-line $1+i\R^+$ is finite.
\begin{definition}
Let $Q^0$ to be the renormalized probability measure of $\bar\nu$ restricted to such configurations.
\end{definition}
Let $(\theta_1,\Gamma_1)$ be a one-point pinned configuration with the law $Q^0$.
We can explore $(\theta_1,\Gamma_1)$  by moving upwards the  half-line $1+i\R^+$, see Figure \ref{2point-pinned}. Let $T=\inf\{t\ge 0, 1+ti \in \theta_1\}$ be the first moment that the half-line intersects $\theta_1$. Let $\tilde A$ be the closure of the union of $[1,1+Ti)$ with all the complete clusters that it intersects. Let $H^0$ be the unbounded connected component of $\H\setminus\tilde A$. Let $\Psi$ be the conformal map from $H^0$ onto $\H$ which sends $0,1+Ti$ to $0,1$ with $\Psi'(0)=1$. Then $(\Psi(\theta_1), \Psi(\Gamma_1))$ is a two-point pinned configuration.
\begin{definition}
Let $Q$ denote the probability measure of the two-point pinned configuration  $(\Psi(\theta_1), \Psi(\Gamma_1))$.
\end{definition}

\begin{figure}[h!]
\centering
\includegraphics[width=\textwidth]{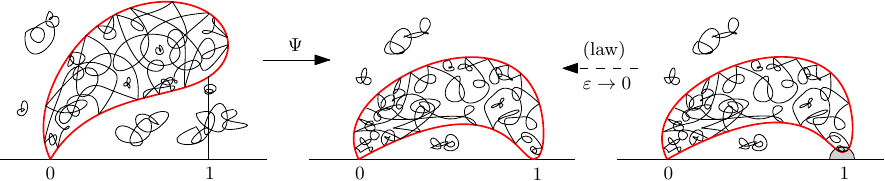}
\caption{The exploration of the one-point pinned configuration and the Definition of the measure $Q$ on two-point pinned complete clusters.}
\label{2point-pinned}
\end{figure}

Similarly to \S\,\ref{sec:2.1}, the law of $\Psi(\theta_1)$ does not depend on the map $\Psi$ and can be obtained as the following limit.
\begin{lemma}\label{lem:3.2}
The law $Q^0$ conditioned on the event $\{\theta_1\cap D(0,\eps)\not=\emptyset\}$ converges as $\eps\to 0$ to $Q$.
\end{lemma}

It can again be proved that the two-point pinned configuration with the law $Q$ still satisfies the CLE-type restriction property.  
Let $(\theta_2, \Gamma_2)$ be a two-point pinned configuration with the law $Q$. Let $\gamma^-$ denote the lower boundary of the two-point pinned complete cluster  $\theta_2$ ($\gamma^-$ has $0$ and $1$ as extremities, see Figure \ref{fig:glued}). 
 We denote by $\Gamma^+$ all the loops in the two-point configuration which are above $\gamma^-$ (including those loops that touch $\gamma^-$). We define $\varphi^-$ to be the conformal map from the unbounded connected component of $\H\setminus \gamma^-$ onto $\H$ that preserves the points $0,1$ and such that $(\varphi^-)'(0)=1$.
\begin{definition}
Let $Q^g$ be the probability measure of the glued configuration $\varphi^-(\Gamma^+)$ and we call $\varphi^-(\theta_2)$  the glued complete cluster.
\end{definition}
The CLE-type loop-restriction property of the two-point pinned configuration implies that
$\varphi^-(\Gamma^+)$ is independent of $\gamma^-$.

\begin{figure}[h!]
\centering
\includegraphics[width=0.78\textwidth]{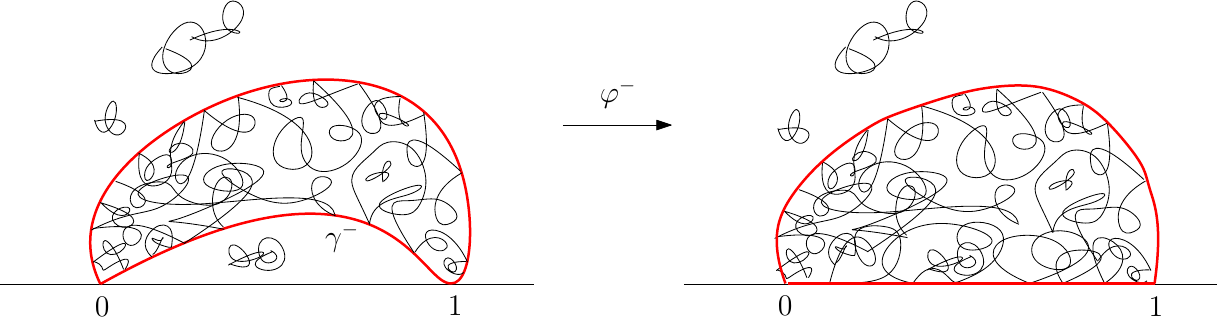}
\caption{The map $\varphi^-$ and the glued configuration.}
\label{fig:glued}
\end{figure}

Note that the measures $Q, Q^g$ are also invariant under all conformal mappings from $\H$ onto itself that leave the points $0,1$ invariant. This is due to the conformal covariance property of $\nu$ and the fact that $Q, Q^g$ are normalized to be probability measures.

Now, similarly to Lemma \ref{lem:3.2}, we want to claim that the measure $Q^g$  is  the measure $Q^0$ conditioned on the event that its one-point pinned complete cluster has $[0,1]$ as a part of its boundary. 
This event has zero probability and we have to make sense of it by taking the limit.
Let $(\theta_1, \Gamma_1)$ be a one-point pinned configuration with the law $Q^0$.
Note that the outer boundary of the pinned complete cluster $\theta_1$ is a simple loop, and hence can be parametrized in the trigonometric orientation by a continuous function $\gamma: \R\to\R^2$ such that $\gamma(0)=0^+, \gamma(\infty)=0^-$. 
Let $T=\inf\{t\ge 0, \gamma(t)\in[1,1+2\delta i]\}$ with the convention that $\inf\emptyset=\infty$. Let $E_\delta$  be the event that the complete cluster $\theta_1$ intersects the line segment $[1,1+2\delta i]$ (so $T<\infty$) and that $\gamma([0,T))$ is a subset of the rectangle $R_\delta=[-\delta,1]\times [0,2\delta]$, namely the event that the first part of the boundary of $\theta_1$ stays close to $[0,1]$ (see Figure \ref{fig:e-delta}).
\begin{lemma}\label{lem:another-glue}
Let $ Q^\delta$ be the measure $Q^0$ conditioned on the event $E_\delta$. Then $Q^\delta$ converges as $\delta\to 0$ to the glued configuration $Q^g$. 
\end{lemma}
\begin{figure}[h!]
\centering
\includegraphics[width=0.45\textwidth]{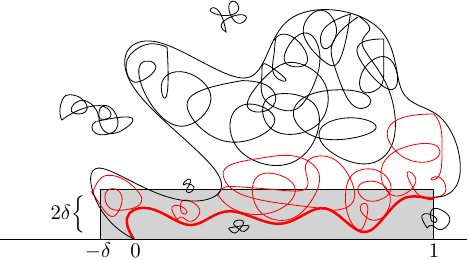}
\caption{One-point pinned configuration conditioned on $E_\delta$.}
\label{fig:e-delta}
\end{figure}

\begin{proof}
Let $(\theta_1^\delta,\Gamma_1^\delta)$ be a one-point pinned configuration with the law $Q^\delta$.
We perform the gluing-down procedure that was explained above: we first explore  $(\theta_1^\delta,\Gamma_1^\delta)$ along $[1,1+2\delta i]$ until the first time we hit the pinned complete cluster and then map it down by $\Psi$ to get the two-point pinned configuration; then we again map it down by $\varphi^-$ to get the glued configuration $(\theta_g^\delta,\Gamma_g^\delta)$. The configuration $(\theta_g^\delta,\Gamma_g^\delta)$ is independent of $\gamma([0,T])$. However the condition on $E_\delta$ is a condition solely on $\gamma([0,T])$. Hence $(\theta_g^\delta,\Gamma_g^\delta)$ has the same law as a glued configuration obtained from an unconditioned one-point pinned configuration. Therefore $(\theta_g^\delta,\Gamma_g^\delta)$ has the law $Q^g$.

Now it is enough to show that  for all $\eps>0$, the probability that the $d^*$ distance between 
$(\theta_g^\delta,\Gamma_g^\delta)$ and  $(\theta_1^\delta,\Gamma_1^\delta)$ be greater than $\eps$ decays to $0$ as $\delta\to 0$. For this, it is enough to show that when $\delta\to 0$, the two conformal maps $\Psi$ and $\varphi^-$ that we applied in the gluing procedure are both very close to the identity with probability tending to $1$. This can be deduced from basic properties of the loop-soup/CLE (note that $\Gamma_1^\delta$ is just an ordinary loop-soup in $\H\setminus\theta^\delta_1$). 
\end{proof}

To end this section, we describe another way to obtain the glued configuration with the law $Q^g$, which is by cutting along the boundary of a one-point pinned complete cluster. This was stated and proved in \cite[Lemma 6.4]{MR2979861} for the CLE case.  We will state in Lemma \ref{lem:edge} the version for the loop-soup case. The proof is the same, hence we do not repeat the details. The idea is that instead of exploring along the vertical half-line $1+i\R^+$, we can explore along any deterministic path in $\H$ with starting point on the boundary, and this leads to the same glued measure. For all simple paths that follow the dyadic grids, conditionally on the fact that the first part of the boundary of the pinned complete cluster is close to this dyadic path, the gluing down procedure is very close to the procedure of cutting along that part of the boundary of the pinned complete cluster.

More precisely, let $(\theta_1,\Gamma_1)$ be sampled with the law $Q^0$. Let $\gamma$ be the outer boundary of the complete cluster $\theta_1$. We give $\gamma$ the following parametrization (see Figure \ref{fig:glued-decomp}).
We orient $\gamma$ ``counterclockwise'', starting from $0^+$ and ending at $0^-$. Fix some $r\in(0,1]$. 
Let $z_0$ be the first intersection point of $\gamma$ with $\{z: |z|=r\}$. We define $b_0:=b_0(\gamma)$ to be the initial part of the loop between $0$ and $z_0$, and we call $e_0:=e_0(\gamma)$ the end-part of the loop $\gamma$ between $z_0$ and $0$.
Let $h_0$ denote the conformal mapping from $\H\setminus b_0$ onto $\H$ normalized by $h_0(z_0)=1, h_0(0^-)=0,h_0(\infty)=\infty$. The image of $e_0$ under $h_0$ is a simple path from $1$ to $0$ in $\H$ that we now call $\xi$. We parametrize $\xi$ in such a way that, the image of $\xi([0,t])$ under the conformal map $z\mapsto 1-1/z$ has half-plane capacity $t$. We denote by $f_t$ some given conformal map from $\H\setminus\xi([0,t])$ onto $\H$ that sends the points $0,\, \xi(t)$ to $0,1$.
Let $h_t=f_t\circ h_0$.
We can now state the following lemma.

\begin{lemma}\label{lem:edge}
For all $t\ge 0$, the image under $h_t$ of $(\theta_1,\Gamma_1)$ has the law $Q^g$.
\end{lemma}

\begin{figure}[h!]
\centering
\includegraphics[width=\textwidth]{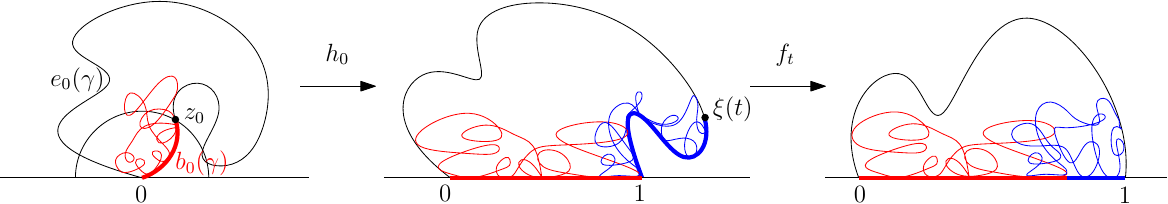}
\caption{The conformal maps $h_0$ and $f_t$.}
\label{fig:glued-decomp}
\end{figure}

\subsection{Decomposition of the glued configuration}

In this section, we are going to decompose the glued configuration with the law $Q^g$. Note that a glued configuration is the union of two collections of loops: the collection $K$ of loops that intersect $[0,1]$ and the collection $\Gamma_0$ of loops that do not intersect $[0,1]$. Our goal is to show that $K$ satisfies conformal restriction property and that the  collection $\Gamma_0$ is just an independent loop-soup in $\H$.
\begin{lemma}[Decomposition of the glued configuration]\label{lem:glued-decomp}
The collections $K$ and $\Gamma_0$ are independent. The law of $\Gamma_0$ is that of a loop-soup in $\H$. 
\end{lemma}
\begin{proof}
We will use the first decomposition lemma for the one-point pinned complete cluster and then combine this with Lemma \ref{lem:another-glue} to conclude.

Let $(\theta_1^\delta,\Gamma_1^\delta)$ be some one-point pinned configuration with the law $Q^\delta$.
Then the union $\theta_1^\delta\cup\Gamma^\delta_1$ is also the union of the following two collections of loops: 
the collection $K^\delta$ of loops that intersect $R_\delta$,
and the collection $\Gamma_0^\delta$ of loops that do not intersect $R_\delta$. 
Note that the measure $Q^\delta$ is in fact equal to $\bar\nu$ restricted to the event $E_\delta$ and then renormalized to have total mass $1$.
By adapting fairly directly the arguments in \cite[Lemma 4]{Qian-Werner}, we can see that the event $E_\delta$ is entirely determined by the loops in $K^\delta$, and hence is independent of the collection $\Gamma_0^\delta$. 
Therefore,
by applying Lemma \ref{lem2} (first decomposition) to $(\theta_1^\delta,\Gamma_1^\delta)$ and  $R_\delta$, we can show that the two sets $K^\delta$ and $\Gamma^\delta_0$ are independent, and moreover,
\begin{itemize}
\item  the collection $K^\delta$ has the law $p^\delta$, where $p^\delta$ is a probability measure defined to be the (infinite) measure $\bar\rho_{R_\delta}$ restricted to the event $E_\delta$ and then renormalized.
\item the collection $\Gamma_0^\delta$ is just a loop-soup in $\H\setminus R_\delta$.
\end{itemize}

Now our goal is to show that the couple $(K^\delta, \Gamma_0^\delta)$ converges in law jointly to the couple $(K,\Gamma_0)$, thus proving the lemma.
Recall that Lemma \ref{lem:another-glue} says that the configuration $K^\delta\cup \Gamma_0^\delta$ converges in law to the configuration $K\cup\Gamma_0$. 
We can then approximate $(K^\delta, \Gamma_0^\delta)$ by the image of $K^\delta\cup \Gamma_0^\delta$ under $X^\eps$, which maps a set $\Lambda$ of loops in $\H$ to the couple $(\Lambda_1, \Lambda_2)$ where $\Lambda_1$ consists of all the loops in $\Lambda$ that intersect the rectangle $R_\eps$ and $\Lambda_2=\Lambda\setminus\Lambda_1$.
 Then the image measure $X^\eps(Q^\delta)$ converges as $\delta\to 0$ to $X^\eps(Q^g)$.
  Then again, as $\eps\to 0$, the measure $X^\eps(Q^g)$ converges to the law of the couple $(K,\Gamma_0)$.
It remains to show that, with high probability, $X^\eps(\theta_1^\delta\cup\Gamma_1^\delta)$ is very close to $(K^\delta, \Gamma_0^\delta)$ with respect to $d^*$, uniformly for all $\delta<\eps$. This is true because the loops that intersect $R_\eps$ but not $R_\delta$ are with high probability very small.
\end{proof}

Now we are going to show that $K$ satisfies conformal restriction. 
\begin{proposition}
The set $K$ satisfies conformal restriction, namely for all closed sets $B, \tilde B$ such that $\H\setminus B, \H\setminus\tilde B \in \mathcal{A}$ and that there is a conformal map $\varphi$ from $\H\setminus B$ onto $\H\setminus\tilde B$, the law of $\varphi(K)$ conditionally on $\{K\cap B=\emptyset\}$ is equal to the law of $K$ conditionally on $\{K\cap \tilde B=\emptyset\}$.
\end{proposition}
\begin{proof}
We continue with the notations of the proof of the previous lemma. We will show that the set $K^\delta$ (with the law $p_\delta$)  satisfies conformal restriction and then conclude by letting $\delta$ go to $0$.

Let $\Lambda$ be a one-point pinned configuration in the support of $\bar\nu$. We denote by $\Lambda^{R_\delta}$ the set of loops in $\Lambda$ that touch $R_\delta$. In the following, we will talk about `conditioning' of infinite measures on some event of finite mass, meaning that we first restrict the measure to this event and then renormalize it to be a probability measure.

Let $p^\delta_{B}$ be the measure $p^\delta$  conditioned on $\{K^\delta\cap B=\emptyset\}$.  
In other words, $p^\delta_{B}$ is also equal to the measure $\bar\rho_{R_\delta}$ first `conditioned' on $E_\delta$ and then conditioned on $\{\Lambda^{R_\delta}\cap B=\emptyset\}$. We can also exchange the order of conditioning and say that $p^\delta_{B}$ is identical to the measure $\bar\rho_{R_\delta}$ first  restricted to the event  $\{\Lambda^{R_\delta}\cap B=\emptyset\}$ and then `conditioned' on $E_\delta$.

By Lemma \ref{lem:restriction1}, the measure $\varphi(\bar\rho_{R_\delta})$ restricted to  $\{\Lambda^{R_\delta}\cap B=\emptyset\}$ is equal to $\varphi'(0)^{-\beta}$ times the measure  $\bar\rho_{\varphi(R_\delta)}$ restricted to  $\{\Lambda^{\varphi(R_\delta)}\cap \tilde B=\emptyset\}$.
Hence $\varphi(p_B^\delta)$ is equal to the measure  $\bar\rho_{\varphi(R_\delta)}$ first restricted to  $\{\Lambda^{\varphi(R_\delta)}\cap \tilde B=\emptyset\}$ and then  `conditioned' on $\varphi(E_\delta)$ where $\varphi(E_\delta)$  is the event that the pre-image by $\varphi$ of a configuration satisfies the event $E_\delta$.
We can again exchange the order of conditioning. Hence $\varphi(p_B^\delta)$ is equal to the measure  $\bar\rho_{\varphi(R_\delta)}$ first  `conditioned' on $\varphi(E_\delta)$ and then conditioned on  $\{\Lambda^{\varphi(R_\delta)}\cap \tilde B=\emptyset\}$.

On the one hand, we know by Lemma \ref{lem:glued-decomp}, that as $\delta\to 0$, $\varphi(p_B^\delta)$ converges to the measure of $\varphi(K)$ conditioned on $\{K\cap B=\emptyset\}$. On the other hand, the measure  $\bar\rho_{\varphi(R_\delta)}$   first  `conditioned' on $\varphi(E_\delta)$ and then conditioned on  $\{\Lambda^{R_\delta}\cap \tilde B=\emptyset\}$ converges in  law to $K$ conditioned on $\{K\cap\tilde B=\emptyset\}$ as $\delta\to 0$, for the same reason as stated in Lemma \ref{lem:another-glue}, \ref{lem:glued-decomp}.
Therefore we have proved the proposition.
\end{proof}

In addition to this proposition, to determine the law of the outer boundary (or the filling) of $K$, it is  \cite{MR1992830} enough to find out the restriction exponent $\alpha$ defined in (\ref{restriction}).
Let $F(K)$ be the filling of the union of all the loops in $K$. Then we have the following proposition.
\begin{proposition}
The set $F(K)$ satisfies one-sided chordal restriction with exponent $\alpha=(6-\kappa)/(2\kappa)$.
\end{proposition}
\begin{proof}
The set $F(K)$ satisfies one-sided chordal restriction with a certain exponent $\alpha>0$.
It was  shown in  \cite{MR3035764} that for a one-sided restriction sample $J$ of exponent $\alpha$ and an independent CLE$_\kappa$,  the outer boundary  of the union of $J$ with all the loops in the CLE$_\kappa$ that it intersects has the law of a SLE$_\kappa(\rho)$ such that $\alpha=(\rho+2)(\rho+6-\kappa)/(4\kappa), \rho>-2$.
Therefore, Lemma \ref{lem:glued-decomp} and Lemma \ref{lem:edge} imply that the curve $h_t (e_0(\gamma))$ is a SLE$_\kappa(\rho)$. 
However we also know \cite{MR2979861} that the curve $h_t (e_0(\gamma))$ is in fact a SLE$_\kappa$ with $\rho=0$. Therefore we must have $\alpha=(6-\kappa)/(2\kappa)$.
\end{proof}

Let us now explain why these  decomposition and conformal restriction statements for the glued configuration in the upper half-plane do imply our main result for explorations of CLE.
For this, using conformal invariance, we need to show that the image under $f_T$ of the loop-soup $\Gamma$ in Figure \ref{fig:intro} is exactly a glued configuration in $\U$.
This will come as a rather direct consequence of the fact derived in \cite{MR2979861} that when one explores/constructs a CLE in a Markovian way, the bubbles that one encounters are distributed like a Poisson point process with intensity given by the (infinite) pinned loop measure $\mu$. 
 More precisely, it is shown in \cite{MR2979861} that one can construct the CLE  in a Markovian way, by composing an ordered countable family  $(f_{t_n})$ from  the connected component of $\U\setminus \xi \gamma_{t_n}$ outside of $\xi\gamma_{t_n}$ onto $\U$, where $(\gamma_{t_n})$ is obtained via a Poisson point process of loops pinned on $1$
  with intensity $\mu$, and  the pinned point $\xi$ is chosen on $\partial
\U$ in a measurable way with respect to the past exploration.
For instance, when exploring along $[-1,1]$, the successive (infinitely many) loops that one meets in Figure \ref{fig:markov-exp} (resp. the clusters in Figure \ref{fig:intro}) are exactly distributed as a Poisson point process with intensity $\mu$ (resp. $\nu$). 
In the setting of Figure \ref{fig:intro}, up to $T$ such that $\gamma_T$ is in the middle of a loop, $f_{\sigma(T)}(\Gamma)$ (where $f_{\sigma(T)}$ is any conformal map from $\U\setminus K_{\sigma(T)}$ onto $\U$) is a one-point pinned configuration ``conditioned'' on the partial discovery of the boundary of its pinned cluster i.e. of   $f_{\sigma(T)}(\gamma(\sigma(T),T))$ (this conditioning makes that $f_{\sigma(T)}(\Gamma)$ is indeed defined under a probability measure). Then we can further map out the curve $f_{\sigma(T)}(\gamma(\sigma(T),T))$ and deduce by Lemma \ref{lem:edge} that $f_T(\Gamma)$ is a glued configuration.

Finally, in order to complete the proof of Theorem \ref {main theorem}, it remains to prove Proposition \ref{coro1} i.e. that the set $\Gamma^b$ satisfies interior restriction.

\begin {proof}[Proof of Proposition \ref{coro1}] Here we go back to the unit disk setting and we continue with the notation in the introduction. We adopt the exploration process with radial parametrization as described in and above Figure \ref{fig:markov-exp}. Note that $S$ is the time at which we discover the (entire) loop surrounding the origin.

By looking at stopping times tending to $S$, we deduce from the main theorem that the set $\Gamma^b$ satisfies interior restriction with the marked point $\varphi_S(1)$. However we can rule out the dependence of the law of $\Gamma^b$ on the point $\varphi_S(1)$ in a similar way as in the proof of Lemma \ref{lem02}. Therefore $\Gamma^b$ satisfies interior restriction.
\end{proof}

\subsection*{Acknowledgements}
The author acknowledges support of the SNF grant SNF-155922. The author is also part of the NCCR Swissmap.
The author is also very grateful to Wendelin Werner for his guidance into this topic and for many insightful suggestions.

\bibliographystyle{plain}
\bibliography{cr} 

\medbreak

Address: D-Math, 
 ETH Z\"urich, 
 R\"amistrasse 101, 8092 Z\"urich, Switzerland
 
 \medbreak

e-mail: {wei.qian@math.ethz.ch}

\end{document}